\documentclass[12pt]{scrartcl}                 
\usepackage{amsmath,amsthm,thm-restate,thmtools,amsfonts,amssymb}
\usepackage[numbers]{natbib}
\usepackage{hyperref,color,comment}

\usepackage{latexsym}

\begin{document}

\newcommand{\out}{\ensuremath{\mathrm{Out}(\mathbb{F}) } }
\newcommand{\aut}{\ensuremath{\mathrm{Aut}(\mathbb{F}) }}
\newcommand{\F}{\ensuremath{\mathbb{F} } }
\newcommand\addtag{\refstepcounter{equation}\tag{\theequation}}
\newtheorem{theorem}{Theorem}[section]
\newtheorem{corollary}[theorem]{Corollary}
\newtheorem{lemma}[theorem]{Lemma}
\newtheorem{proposition}[theorem]{Proposition}
\newtheorem{remark}[theorem]{Remark}
\newtheorem{definition}[theorem]{Definition}
\newtheorem*{thma}{Theorem A}

\title{Relative hyperbolicity of Free-by-cyclic extensions}
\author{Pritam Ghosh}


\date{}

\maketitle

\begin{abstract}
 Given a finitely generated free group $\F$ of $\mathsf{rank}(\F)\geq 3$, we show that the mapping torus of $\phi$ 
 is (strongly) relatively hyperbolic if $\phi$ is exponentially growing. As a corollary of our work, we give a new proof 
 of Brinkmann's theorem which proves that the mapping torus of an atoroidal outer automorphism is hyperbolic. We also give a new proof 
 of the Bridson-Groves theorem that the mapping torus of a free group automorphism satisfies the quadratic isoperimetric inequality.

 Our work also answers a question asked by Minasyan-Osin: 
 the mapping torus of an outer autmomorphism is not virtually acylindrically hyperbolic if and only if $\phi$ has finite order. 
\end{abstract}

\section{Introduction}

Fix a finitely generated free group $\F$ with $\mathsf{rank}(\F)\geq 3$. Any element $\phi\in\out$ (the outer automorphism 
group of $\F$) gives us a short exact sequence $$1 \to \F \to \Gamma \to \langle \phi \rangle \to 1$$ where the 
extension group $\Gamma$ is referred to as the \textit{mapping torus} of $\phi$. Significant work has been done in understanding
geometry of $\Gamma$. Bestvina-Feighn-Handel first proved that the mapping torus is hyperbolic if $\phi$ is fully irreducible and 
atoroidal \cite[Theorem 5.1]{BFH-97}. It follows from work of Bestvina-Feighn \cite{BF-92} and Brinkmann \cite{Br-00} 
that $\Gamma$ is hyperbolic if and only if $\phi$ does not have any periodic conjugacy classes (i.e. $\phi$ is \textit{atoroidal}).
We contribute to this study by showing 

\newtheorem*{thmb}{Theorem \ref{main}}
 \begin{thmb}
If $\phi\in\out$, then the extension group $\Gamma$ in the sequence $1 \to \F \to \Gamma \to \langle \phi \rangle \to 1$ 
 is strongly relatively hyperbolic if and only if $\phi$ is exponentially growing.
 \end{thmb}

 One interesting observation that comes out of this result is connected to the mapping class group theory of surfaces with boundaries. 
 It is well known that pseudo-Anosov maps of such surfaces give us (strongly) relatively hyperbolic extension groups. Our work here 
 shows that pseudo-Anosov's are not the only type of maps which give relatively hyperbolic extensions.

The ``only if'' direction follows from the work of Makura \cite{Mak-02}. The ``if'' direction is our main contribution to the above theorem. 
We prove this in Proposition \ref{relhypext} where we show that $\Gamma$ will be (strongly) hyperbolic relative to a collection of subgroups which 
correspond to the mapping torus of (representatives of) components of the nonattracting subgroup system of certain attracting lamination of $\phi$. 
This connection between nonattracting subgroup system and relative hyperbolicity was first established in \cite{self2}. This connection is very natural in the 
sense that a conjugacy class is attracted to some lamination under iteration of $\phi$ if and only if it is not carried by the nonattracting subgroup 
system; and being attracted to any attracting lamination implies exponential growth for the conjugacy class.

As a corollary of Proposition \ref{relhypext}, we give a new proof of Brinkmann's result by a simple inductive argument (inducting on the rank of 
the subgroups of $\F$ chosen as representatives of components of the nonattracting subgroup system) on the peripheral subgroups 
obtained in Proposition \ref{relhypext}.

\newtheorem*{thmd}{Corollary \ref{atoroidal}}
\begin{thmd}
 $\phi\in\out$ is atoroidal if and only if the extension group $\Gamma$, in the short exact sequence $1\to\F\to\Gamma\to\langle\phi\rangle\to 1 $, 
  is a hyperbolic group.
\end{thmd}

We then combine our work with a result of Button-Kropholler \cite{BuK-16} to answer a question asked by 
Minasyan and Osin in \cite[Problem 8.2]{MinOs-15}. 

\newtheorem*{thmc}{Corollary \ref{coro}}
 \begin{thmc}
If $\phi\in\out$, then the extension group $\Gamma$ in the sequence $1 \to \F \to \Gamma \to \langle \phi \rangle \to 1$ 
 is not virtually acylindrically hyperbolic if and only if $\phi$ has finite order.
 \end{thmc}

 The proof of Theorem \ref{main} uses the completely split train track theory from Feighn-Handel's Recognition theorem 
 work \cite{FH-11} and the weak attraction theory from Handel-Mosher's Subgroup decomposition for $\out$ body of work \cite{HM-13c}.

 \textbf{Idea of proof: }
 Given an exponentially growing $\phi$, one always has an attracting lamination $\Lambda^+$ associated to it. The 
 nonattracting subgroup system $\mathcal{A}_{na}(\Lambda^+)$ is a malnormal subgroup system that carries all conjugacy classes 
 which are not attracted to $\Lambda^+$ under iterations of $\phi$ (see section 2.1). 
 
 We perform a partial electrocution of $\Gamma$ with respect to a collection of certain lifts of the components of $\mathcal{A}_{na}(\Lambda^+)$
 and form a new metric space $(\widehat{\Gamma}, {|\cdot|}_{el})$. After this we use the weak attraction theorem 
 to show the under iteration of either $\phi$ or $\phi^{-1}$, we gain enough legality \ref{legality} to proceed 
(primarily following the technique in Bestvina-Feighn-Handel's work in \cite{BFH-97}) to prove that we have flaring \ref{strictflare}. 
This combined with the Mj-Reeves strong combination theorem \cite[Theorem 4.6]{MjR-08} proves that $\Gamma$ will be strongly relatively 
hyperbolic in Proposition \ref{relhypext}.

When $\phi$ is atoroidal, a simple inductive argument by repeatedly applying Proposition \ref{relhypext} on the peripheral subgroups which are being 
electrocuted to form the 
coned-off graph, shows that $\Gamma$  will be hyperbolic thus giving a new proof of Brinkmann's theorem.
 
In the last section we show an application of our work which generalizes a theorem of \cite{BFH-97} and 
appears in form of the following theorem:

\newtheorem*{thme}{Theorem \ref{free-by-free}}
 \begin{thme}
 Let $\phi, \psi$ be outer automorphisms which satisfy the standing assumptions \ref{SA}. Then there exists some 
  $M>0$ such that for every $m, n \geq M$ the group $Q:=\langle \phi^m, \psi^n \rangle$ is a free group of rank $2$ and 
  the extension group $\F\rtimes \widetilde{Q}$ is hyperbolic for any lift $\widetilde{Q}$ of $Q$.
 \end{thme}
 
 The proof of this theorem proceeds by proving a version of Mosher's 3-of-4 stretch lemma \ref{34} in this context and using it together with the 
 Bestvina-Feighn combination theorem. An alternative proof of this theorem can be obtained by a recent work of Uyanik \cite{Uya-17}.
 
Hyperbolic extensions of free groups have also been produced by Dowdall-Taylor \cite{DT-18} in their work on convex cocompact subgroups of $\out$ and 
by Uyanik in a recent work \cite{Uya-17}. Uyanik does not have any assumptions of fully-irreducibilty on the elements of the quotient group. 
 Theorem \ref{free-by-free}  gives a  new class of examples of free-by-free hyperbolic extensions and we 
 hope that this will be useful in the future when further research is done to give even  weaker conditions on $Q$ so that 
 the extension group is hyperbolic. This result is a significant generalization of a theorem of Bestvina-Feighn-Handel \cite[Theorem 5.2]{BFH-97}
  where they prove a similar result by assuming $\phi,\psi$ to be fully irreducible and atoroidal.
  
  As an application, we use our work together with the polynomial growth case of Bridson-Groves' theorem \cite{BrG-10}, 
  to give a new proof of the general case of the  Bridson-Groves theorem \cite{BrG-10} which shows that the mapping torus of any free group 
  automorphism satisfies the quadratic isoperimetric inequality, implying that the conjugacy problem is solvable for such groups.

  \newtheorem*{thmk}{Theorem \ref{quad}}
  \begin{thmk}
   The mapping tori of a free group automorphism satisfies the quadratic isoperimetric inequality.
  \end{thmk}

 \textbf{Acknowledgement:} The author would like to sincerely thank Pranab Sardar for several helpful discussions, feedback and for pointing out 
 Makura's work to the author. Also thanks to Derrik Wigglesworth for pointing out a flaw in the proof of conjugacy flaring and other helpful 
 comments from the first draft of this paper. Thanks to Jason Behrstock and Mark Hagen for their feedback, encouragement and many helpful discussions.
 
  The author was supported by SERB N-PDF grant PDF/2015/000038 during this project.

\section{Preliminaries }

We recall some of the basic notions and tools used in the study of $\out$. The definitions and results presented in this section 
have been developed over a significant period of time in \cite{BH-92}, \cite{BFH-00}, \cite{FH-11} by Bestvina, Feighn and Handel. 

\begin{description}

 \item \textbf{Marked graphs, circuits and path: } A \textit{marked graph} is a graph $G$ which has no valence 1 vertices and equipped with 
 a homotopy equivalence to the rose $m: G\to R_n$ (where $n = \text{rank}(\F)$). The fundamental group of $G$ therefore can be identified with 
 $\F$ up to inner automorphism. A \textit{circuit} in a marked graph is an immersion (i.e. locally injective and continuous map)
 of $S^1$ into $G$. The set of circuits in $G$ can be identified 
 with the set of conjugacy classes in $\F$. Similarly a path is an immersion of the interval $[0, 1]$ into $G$. Given any continuous map 
 from $S^1$ or $[0,1]$ it can be \textit{tightened} to a circuit or path, meaning the original map is freely homotopic to a locally injective and continuous 
 map from the respective domains. In fact, given any homotopically nontrivial and continuous map from $S^1$ to $G$, it can be tightened to a 
  unique circuit. We shall not distinguish between circuits or paths that differ by a homeomorphism of their respective domains.
 
 \item \textbf{Free factor systems:} Given a collection of finitely generated subgroups of $\F$, say $F^1, F^2, ..., F^k$, the collection of their conjugacy classes 
 denoted by $\{[F^1], [F^2],... , [F^k]\}$ is called a \textit{subgroup system}. We say that a circuit $[c]$ is carried by this subgroup 
 system if and only if there exists a representative $H^i$ of some $[F^i]$ such that $c \in H^i$.  The subgroup system is called a 
 \textit{malnormal subgroup system}  if $H^i \cap H^j = \emptyset$ for every $i\neq j$, where $H^i$ is any representative of $[F^i]$.
 
 A subgroup system $\{[F^1], [F^2],... , [F^k]\}$ is called a \textit{free factor system} if $\F= F^1 \ast F^2 \ast.....\ast F^k \ast B$, with $B$ possibly 
 trivial. Note that every free factor system is always a malnormal subgroup system. Note that given any subgraph $H\subset G$, the fundamental groups 
 of the noncontractible  components of $H$ gives rise to a free factor system.
 
 \item  \textbf{Weak topology: }
 Given any finite graph $G$, let $\widehat{\mathcal{B}}(G)$ denote the compact space of equivalence classes of circuits in $G$ and paths in $G$, whose endpoints (if any) are vertices of $G$. We give this space the \textit{weak topology}.
Namely, for each finite path $\gamma$ in $G$, we have one basis element $\widehat{N}(G,\gamma)$ which contains all paths and circuits in $\widehat{\mathcal{B}}(G)$ which have $\gamma$ as its subpath.
Let $\mathcal{B}(G)\subset \widehat{\mathcal{B}}(G)$ be the compact subspace of all lines in $G$ with the induced topology. One can give an equivalent description of $\mathcal{B}(G)$ following \cite{BFH-00}.
A line is completely determined, upto reversal of direction, by two distinct points in $\partial \mathbb{F}$, since there only one line that joins these two points. 
We can then induce the weak topology on the set of lines coming from the Cantor set $\partial \mathbb{F}$. More explicitly,
let $\widetilde{\mathcal{B}}=\{ \partial \mathbb{F} \times \partial \mathbb{F} - \vartriangle \}/(\mathbb{Z}_2)$, where $\vartriangle$ is the diagonal and $\mathbb{Z}_2$ acts by interchanging factors. We can put the weak topology on
$\widetilde{\mathcal{B}}$, induced by Cantor topology on $\partial \mathbb{F}$. The group $\mathbb{F}$ acts on $\widetilde{\mathcal{B}}$ with a compact but non-Hausdorff quotient space $\mathcal{B}=\widetilde{\mathcal{B}}/\mathbb{F}$.
The quotient topology is also called the \textit{weak topology}. Elements of $\mathcal{B}$ are called \textit{lines}. A lift of a line $\gamma \in \mathcal{B}$ is an element  $\widetilde{\gamma}\in \widetilde{\mathcal{B}}$ that
projects to $\gamma$ under the quotient map and the two elements of $\widetilde{\gamma}$ are called its endpoints.

One can naturally identify the two spaces $\mathcal{B}(G)$ and $\mathcal{B}$ by considering a homeomorphism between the two Cantor sets $\partial \mathbb{F}$ and set of ends of universal cover of $G$ , where $G$ is a marked graph.
$\out$ has a natural action on  $\mathcal{B}$. The action comes from the action of Aut($\mathbb{F}$) on $\partial \mathbb{F}$. Given any two marked graphs $G,G'$ and a homotopy equivalence $f:G\rightarrow G'$ between them, the induced map
$f_\#: \widehat{\mathcal{B}}(G)\rightarrow \widehat{\mathcal{B}}(G')$ is continuous and the restriction $f_\#:\mathcal{B}(G)\rightarrow \mathcal{B}(G')$ is a homeomorphism. With respect to the identification
$\mathcal{B}(G)\approx \mathcal{B}\approx \mathcal{B}(G')$, if $f$ preserves the marking then $f_{\#}:\mathcal{B}(G)\rightarrow \mathcal{B}(G')$ is equal to the identity map on $\mathcal{B}$. When $G=G'$, $f_{\#}$ agree with their
homeomorphism $\mathcal{B}\rightarrow \mathcal{B}$ induced by the outer automorphism associated to $f$.

A line(path) $\gamma$ is said to be \textit{weakly attracted} to a line(path) $\beta$ under the action of $\phi\in\out$, if the $\phi^k(\gamma)$ converges to $\beta$ in the weak topology. This is same as saying, for any given finite subpath of $\beta$, $\phi^k(\gamma)$
contains that subpath for some value of $k$; similarly if we have a homotopy equivalence $f:G\rightarrow G$,  a line(path) $\gamma$ is said to be \textit{weakly attracted} to a line(path) $\beta$ under the action of $f_{\#}$ if the $f_{\#}^k(\gamma)$ converges to $\beta$
in the weak topology.

 \item \textbf{EG strata and NEG strata:}
  
 A \textit{filtration} of a marked graph $G$ is a strictly increasing sequence of subgraphs 
 $G_0 \subset G_1 \subset \cdots \subset G_k = G$, each with no isolated vertices. 
 The individual terms $G_k$ are called \textit{filtration elements}, and if $G_k$ is a core graph (i.e. a graph without valence 1 vertices)
 then it is called a 
 \textit{core filtration element}. The subgraph $H_k = G_k \setminus G_{k-1}$ together with the vertices which occur as endpoints of edges in 
 $H_k$ is called the \textit{stratum of height $k$}.
 The \textit{height}\index{height} of subset of $G$ is the minimum $k$ such that the subset is contained in $G_k$. 
 The height of a map to $G$ is the height of the image of the map. 
 A \textit{connecting path} of a stratum $H_k$ is a nontrivial finite path $\gamma$ of height $< k$ whose endpoints are contained in 
 $H_k$.

Given a topological representative $f : G \to G$ of $\phi \in \out$, we say that $f$ \textit{respects} the filtration or 
that the filtration is \textit{$f$-invariant} if $f(G_k) \subset G_k$ for all $k$. If this is the case then we also say that the 
filtration is \textit{reduced} if for each free factor system $\mathcal{F}$ which is invariant under $\phi^i$ for some $i \ge 1$, if 
$[\pi_1 G_{r-1}] \sqsubset \mathcal{F} \sqsubset [\pi_1 G_r]$ then either $\mathcal{F} = [\pi_1 G_{r-1}]$ or $\mathcal{F} = [\pi_1 G_r]$.

Given an $f$-invariant filtration, for each stratum $H_k$ with edges $\{E_1,\ldots,E_m\}$, define
the \textit{transition matrix} of $H_k$ to be the square matrix whose $j^{\text{th}}$ column records the number of times 
$f(E_j)$ crosses the other edges. If $M_k$ is the zero matrix then we say that $H_k$ is a \textit{zero stratum}. 
If $M_k$ irreducible --- meaning that for each $i,j$ there exists $p$ such that the $i,j$ entry of the $p^{\text{th}}$ 
power of the matrix is nonzero --- then we say that $H_k$ is \textit{irreducible}; and if one can furthermore choose $p$ independently 
of $i,j$ then $H_k$ is \textit{aperiodic}. Assuming that $H_k$ is irreducible, by  Perron-Frobenius theorem, the matrix $M_k$ 
a unique eigenvalue $\lambda \ge 1$, called the \textit{Perron-Frobenius eigenvalue}, for which some associated eigenvector has 
positive entries: if $\lambda>1$ then we say that $H_k$ is an \textit{exponentially growing} or EG stratum; whereas if $\lambda=1$ 
then $H_k$ is a \textit{nonexponentially growing} or NEG stratum.

 \item \textbf{Relative train track maps: }
 Given $\phi\in\out$ a \textit{topological representative} is a homotopy equivalence $f:G\rightarrow G$ where $G$ is a marked graph, 
$f$ takes vertices to vertices and edges to edge-paths and $\overline{\rho}\circ f \circ \rho: R_n \rightarrow R_n$ represents $\phi$. A nontrivial path $\gamma$ in $G$ 
is a \textit{periodic Nielsen path} if there exists a $k$ such that $f^k_\#(\gamma)=\gamma$;
the minimal such $k$ is called the period and if $k=1$, we call such a path \textit{Nielsen path}. A periodic Nielsen path is \textit{indivisible} if it cannot 
be written as a concatenation of two or more nontrivial periodic Nielsen paths.

Given a subgraph $H\subset G$ let $G\setminus H$ denote the union of edges in $G$ that are not in $H$.

Given a marked graph $G$ and a homotopy equivalence $f:G\rightarrow G$ that takes edges to paths, one can define a new map $Tf$ by setting $Tf(E)$ 
to be the first edge in the edge path associated to $f(E)$; similarly let $Tf(E_i,E_j) = (Tf(E_i),Tf(E_j))$. So $Tf$ is a map that takes turns to turns. We say that a 
nondegenerate turn is \textit{illegal} if for some iterate of $Tf$ the turn becomes degenerate; otherwise the
 turn is legal. 

 \textbf{Relative train track map.} Given $\phi\in \out$ and a topological representative $f:G\rightarrow G$ with a filtration 
 $G_0\subset G_1\subset \cdot\cdot\cdot\subset G_k$ which is preserved by $f$, 
 a path is said to be \textit{legal} if it contains only legal turns and it is $r-$legal if it is of height $r$ and all its illegal turns are in $G_{r-1}$.
 The topological representative $f$ is a train relative train track map if every stratum is either a zero stratum or irreducible stratum and 
 additionally the following conditions are satisfied for every EG stratum $H_r$ :
 \begin{enumerate}
  \item $f$ maps r-legal paths to legal r-paths.
  \item If $\gamma$ is a path in $G$ of height $< r$ with its endpoints in $H_r$ then $f_\#(\gamma)$ has its end points in $H_r$.
  \item If $E$ is an edge in $H_r$ then $Tf(E)$ is an edge in $H_r$
 \end{enumerate}

 \item \textbf{Completely split train track maps: }
 Given relative train track map $f:G\rightarrow G$, \textit{splitting} of a line, path or a circuit $\gamma$ is a decomposition of $\gamma$ into subpaths $....\gamma_0\gamma_1 .....\gamma_k....  $ 
 such that for all $i\geq 1$ the path $f^i_\#(\gamma) =  .. f^i_\#(\gamma_0)f^i_\#(\gamma_1)...f^i_\#(\gamma_k)...$
 The terms $\gamma_i$ are called the \textit{terms} of the splitting of $\gamma$.

 Given two linear edges $E_1,E_2$ and a root-free closed Nielsen path $\rho$ such that $f_\#(E_i) = E_i.\rho^{p_i}$ then we say that $E_1,E_2$ are said to be in the \textit{same linear family} and any path of the form $E_1\rho^m\overline{E}_2$ for some integer $m$ is called an \textit{exceptional path}.

 \textbf{Complete splittings:} A splitting of a path or circuit $\gamma = \gamma_1\cdot\gamma_2......\cdot \gamma_k$ is called 
 \textit{complete splitting} if each term $\gamma_i$ falls into one of the following categories:
 \begin{itemize}
  \item $\gamma_i$ is an edge in some irreducible stratum.
  \item $\gamma_i$ is an indivisible Nielsen path.
  \item $\gamma_i$ is an exceptional path.
  \item $\gamma_i$ is a maximal subpath of $\gamma$ in a zero stratum $H_i$ and $\gamma_i$ is taken.
 \end{itemize}

 \textbf{Completely split improved relative train track maps}. A \textit{CT} or a completely split improved relative train track maps are topological representatives with particularly nice properties. But CTs do not exist for all outer automorphisms. 
 Only the \textit{rotationless} (see Definition 3.13 \cite{FH-11}) outer automorphisms are guaranteed to have a CT representative
 as has been shown in the following Theorem from \cite{FH-11}(Theorem 4.28).
 \begin{lemma}
  For each rotationless $\phi\in \out$ and each increasing sequence $\mathcal{F}$ of $\phi$-invariant free factor systems, there exists a CT $f:G\rightarrow G$ that is a topological
  representative for $\phi$ and $f$ realizes $\mathcal{F}$.
 \end{lemma}

 The following results are some properties of CT's defined in Recognition theorem work of Feighn-Handel in \cite{FH-11}.
 We will state only the ones we need here.
 \begin{enumerate}\label{CT}
  \item \textbf{(Rotationless)} Each principal vertex is fixed by $f$ and each periodic direction at a principal vertex is fixed by $Tf$.
  \item \textbf{(Completely Split)} For each edge $E$ in each irreducible stratum, the path $f(E)$ is completely split.
  \item \textbf{(vertices)} The endpoints of all indivisible Nielsen paths are vertices. The terminal endpoint of each nonfixed NEG edge is principal.
  \item \textbf{(Periodic edges)} Each periodic edge is fixed.
  \item \textbf{(Zero strata)} Each zero strata $H_i$ is contractible and enveloped by a EG strata
  $H_s, s>i$, such that every edge of $H_i$ is a taken in $H_s$. Each vertex of $H_i$ is contained in $H_s$ and
  link of each vertex in $H_i$ is contained in $H_i \cup H_s$.
  \item \textbf{(Linear Edges)} For each linear edge $E_i$ there exists a root free indivisible Nielsen
  path $w_i$ such that $f_\#(E_i) = E_i w^{d_i}_i$ for some $d_i \neq 0$.
  \item \textbf{(Nonlinear NEG edges)} \cite[Lemma 4.21]{FH-11} Each non-fixed NEG stratum $H_i$ is a single edge with its
  NEG orientation and has a splitting $f_\#(E_i) = E_i\cdotp u_i$, where $u_i$ is a closed nontrivial  completely split
  circuit and is an indivisible Nielsen path if and only if $H_i$ is linear.
   \end{enumerate}

 We shall call any nonfixed, nonlinear NEG edge a \textit{superlinear NEG edge}. The advantage of using CT maps rather than using the 
 regular relative train track maps is the greater control that we get when we iterate the train track maps, which is very much a necessity 
 here.

 \item \textbf{Attracting laminations: }
 A closed subset of $\mathcal{B}=\mathcal{B}(\F)$ is called a \textit{lamination}. An element of a lamination is called a \textit{leaf}.
 The action of $\out$ on $\mathcal{B}$ induces an action on the set of laminations. 

For each marked graph $G$ the homeomorphism $\mathcal{B} \approx \mathcal{B}(G)$ induces a bijection between $\F$-laminations and 
closed subsets of $\mathcal{B}(G)$. The closed subset of $\mathcal{B}(G)$ corresponding to a lamination $\Lambda \subset \mathcal{B}$ 
is called the \textit{realization} of $\Lambda$ in~$G$; we will generally use the same notation $\Lambda$ for its realizations in marked graphs. 
Also, we occasionally use the term lamination to refer to $\F$-invariant, closed subsets of $\widetilde{\mathcal{B}}$; 
the orbit map $\widetilde{\mathcal{B}} \mapsto \mathcal{B}$ puts these in natural bijection with laminations in~$\mathcal{B}$. 


Given $\phi \in \out$ and a lamination $\Lambda \subset \mathcal{B}$, we say that $\Lambda$ is an \textit{attracting lamination} 
for $\phi$ if there exists a leaf $\ell \in \Lambda$ satisfying the following: 
\begin{itemize}
 \item $\Lambda$ is the weak closure of $\ell$
 \item $\ell$ is a birecurrent
 \item $\ell$ is not the axis of the conjugacy class of a generator of a rank one free factor of $F_n$
 \item there exists $p \ge 1$ and a weak open set $U \subset \mathcal{B}$ such that $\phi^p(U) \subset U$ and such that $\{\phi^{kp}(U) 
\ni k \ge 1\}$ is a weak neighborhood basis of $\ell$
\end{itemize}

Any such leaf $\ell$ is called a \textit{generic leaf} of $\Lambda$, 
and any such neighborhood $U$ is called an \textit{attracting neighborhood of $\Lambda$ for the action of $\phi^p$}. 
Let $\mathcal{L}(\phi)$ denote the set of attracting laminations for~$\phi$.

Bestvina-Feighn-Handel \cite{BFH-00} showed that there is bijection between the elements of $\mathcal{L}(\phi)$ and the exponentially growing 
strata of $G$. In fact, they showed that given a relative train track map, there is a unique way of associating each attracting lamination to 
an exponentially growing stratum of $G$. Moreover, they established a bijection between the sets $\mathcal{L}(\phi)$ and $\mathcal{L}(\phi^{-1})$ by using the 
notion of \textit{free factor support}. In this particular case, one can think of a free factor support of an attracting lamination, 
$\Lambda^+\in\mathcal{L}(\phi)$, to be the conjugacy class of the smallest (in terms of subgroup inclusion) free factor that carries 
$\Lambda^+$. Two laminations $\Lambda^+\in\mathcal{L}(\phi)$ and $\Lambda^-\in\mathcal{L}(\phi^{-1})$ are said to be \textit{dual} if and only if they 
have the same free factor support. In \cite{BFH-00} it is shown that duality induces a bijection between the sets $\mathcal{L}(\phi)$ and 
$\mathcal{L}(\phi^{-1})$.

Given a rotationless outer automorphism $\phi\in\out$ we list a few properties which will be important for us.

\begin{lemma}\cite[Lemma 3.30, Corollary 3.31]{FH-11}
 For a rotationless $\phi\in\out$ the following are true:
 \begin{enumerate}
  \item If $F$ is a $\phi$  invariant free factor then ${\phi|}_{F}$ is rotationless.
  \item Each $\Lambda^+\in\mathcal{L}(\phi)$ is invariant under $\phi$.
  \item Every free factor, conjugacy class which is periodic under $\phi$ is fixed by $\phi$.
  
 \end{enumerate}

\end{lemma}
 
\end{description}

\subsection{Relatively hyperbolic groups: }
Given a group $\Gamma$ and a collection of subgroups $H_\alpha < \Gamma$, the \textit{coned-off Cayley graph} of $\Gamma$ or 
  the \textbf{electric space} of $\Gamma$ relative to the collection $\{H_\alpha\}$ is a metric space which consists of the 
  Cayley graph of $\Gamma$ and a collection of vertices $v_\alpha$ (one for each $H_\alpha$) such that each point of 
  $H_\alpha$ is joined to (or coned-off at) $v_\alpha$ by an edge of length $1/2$. The resulting metric space is 
  denoted by $(\widehat{\Gamma}, {|\cdot|}_{el})$.

  A group $\Gamma$ is said to be \textit{(weakly) hyperbolic} relative to a \emph{finite} collection of \emph{finitely-generated} subgroups 
  $\{H_\alpha\}$ if $\widehat{\Gamma}$ is a $\delta-$hyperbolic metric space, in the sense of Gromov.
  $\Gamma$ is said to be strongly hyperbolic relative to the collection $\{H_\alpha\}$ if the coned-off 
  space $\widehat{\Gamma}$ is weakly hyperbolic relative to $\{H_\alpha\}$ and it satisfies the 
  \textit{bounded coset penetration} property (see \cite{Fa-98}). 
  But 
this bounded coset penetration property is a very hard condition to check for random groups $\Gamma$. 
However it is well known that if the group $G$ is hyperbolic and the collection of subgroups $\{H_\alpha\}$  is 
\textit{mutually malnormal} and \textit{quasiconvex} then $\widehat{H}$ is strongly relatively hyperbolic. We shall be using this result 
for our constructions here.

In our main result, Proposition \ref{relhypext}, the bounded coset penetration property is verified by the cone-bounded hallways strictly 
flare condition due to \cite{MjR-08}. It was shown in \cite{MjR-08} that this flaring property establishes a condition (namely, \textit{mutual coboundedness})
due to Bowditch \cite{Bow-97} which implies the strong relative hyperbolicity.

\section{Exponential growth case}
\label{exp}

 We begin this section by recalling the construction of the nonattracting subgroup system and list few of the properties we will be using.
 We give the definitions and some results about the two key 
 concepts, nonattracting subgroup system and weak attraction theorem,
 from the subgroup decomposition work of Handel and Mosher\cite[Section 1 and Theorem F]{HM-13c} which are central to the proofs in this paper. 
 
 Recall from \cite{FH-11} that there exists some $K>0$ such that given any $\phi\in\out$, $\phi^K$ is \textit{rotationless}. 
 Hence given any $\phi$ we may pass on to a rotationless power to make use of the rich CT structure. We show that the mapping torus of 
 $\phi^K$ is relatively hyperbolic. Then using Drutu's work \cite{Dr-09} we conclude that the mapping torus of $\phi$ is also relatively hyperbolic, since 
 mapping torus of $\phi^K$ is quasi-isometric to mapping torus of $\phi$.
 
\textbf{Topmost lamination: } For an outer automorphism $\phi$, we call an attracting lamination $\Lambda$ to be \textit{topmost} 
if there are no attracting lamination of $\phi$ that contain $\Lambda$ as a proper subset. 

It is easy to see that for any exponentially growing outer automorphism, if we choose a relative train track map, then the attracting lamination associated to 
the highest stratum is always topmost. So every exponentially growing outer automorphism has at least one topmost attracting lamination.
From \cite[Corollary 6.0.1]{BFH-00} we know that if $\Lambda^\pm$ is a dual lamination pair for $\phi$, then $\Lambda^+$ is topmost lamination for 
$\phi$ if and only if $\Lambda^-$ is topmost for $\phi^{-1}$.

 \subsection{Nonattracting subgroup system}
 The \textit{nonattracting subgroup system} of an attracting lamination contains information about lines and circuits which are not attracted to the lamination.
This is a crucial construction from the train-track theory that lies in the heart of our proof here. First introduced by Bestvina-Feighn-Handel in their Tit's alternative work \cite{BFH-00}, 
it was later studied in more details by Handel-Mosher in \cite{HM-13c}. We urge the reader unfamilar with this construction to look into \cite{HM-13c} where it has been 
explored in great detail. 

 The construction of the \textbf{nonattracting subgraph} is as following:

 Suppose $\phi\in\out$ is rotationless and $f:G\rightarrow G$ is a CT representing $\phi$ such that $\Lambda^+_\phi$ is an invariant attracting lamination which corresponds to the EG stratum $H_s\in G$.
 The \textit{nonattracting subgraph Z} of $G$ is defined as a union of irreducible stratas $H_i$ of $G$ such that no edge in $H_i$ is weakly attracted to $\Lambda^+_\phi$. This is equivalent to saying that a strata $H_r\subset G\setminus Z$ if and only if there exists $k\geq 0$
 some term in the complete splitting of $f^k_\#(E_r)$ is an edge in $H_s$. 
 
 Define the path $\widehat{\rho}_s$ to be trivial path at any chosen vertex if there does not exist any indivisible Nielsen path of height $s$, otherwise $\widehat{\rho}_s$ 
 is the unique closed indivisible path of height $s$ (from definition of stable train track maps).

\textbf{The groupoid  $\langle \mathcal{Z}, \widehat{\rho}_s \rangle$ - } Let $\langle \mathcal{Z}, \widehat{\rho}_s \rangle$ be the set of lines, rays, circuits and finite paths in $G$ which can be written as a concatenation of subpaths, each of which is an edge in $Z$, the path $\widehat{\rho}_s$ or its inverse. Under the operation of tightened concatenation of paths in $G$,
this set forms a groupoid (Lemma 5.6, [\cite{HM-13c}]). We say that a path, circuit, ray or line is \textit{\textbf{carried by}} $\langle\mathcal{Z}, \hat{\sigma}\rangle$ if 
it can be written as a concatenation of paths in $\langle\mathcal{Z}, \hat{\sigma}\rangle$.


Define the graph $K$ by setting $K=Z$ if $\widehat{\rho}_s$ is trivial and let $h:K \rightarrow G$ be the inclusion map. Otherwise define an edge $E_\rho$ representing the domain of the Nielsen path $\rho_s:E_\rho \rightarrow G_s$, and let $K$ be the disjoint union of $Z$ and $E_\rho$ with the following identification.
 Given an endpoint $x\in E_\rho$, if $\rho_s(x)\in Z$ then identify $x\sim\rho_s(x)$.Given distinct endpoints $x,y\in E_\rho$, if $\rho_s(x)=\rho_s(y)\notin Z$ then identify $x \sim y$. In this case define $h:K\rightarrow G$ to be the inclusion map on $K$ and the map $\rho_s$ on $E_\rho$. It is not difficult to see that the map $h$ is an immersion.
 Hence restricting $h$ to each component of $K$, we get an injection at the level of fundamental groups. The \textbf{nonattracting subgroup system} $\mathcal{A}_{na}(\Lambda^+_\phi)$ is defined to be the subgroup system defined by this immersion.

We will leave it to the reader to look it up in \cite{HM-13c} where it is explored in details. We however list some key properties which we will be using and justifies the importance of this
subgroup system.
\begin{lemma}(\cite{HM-13c}- Lemma 1.5, 1.6)
\label{NAS}
 \begin{enumerate}
 \item A line or conjugacy class is carried by $\langle\mathcal{Z}, \hat{\sigma}\rangle$ if and only if it is carried by $\mathcal{A}_{na}(\Lambda^+_\phi)$.
  \item The set of lines carried by $\mathcal{A}_{na}(\Lambda^+_\phi)$ is closed in the weak topology.
  \item A conjugacy class $[c]$ is not attracted to $\Lambda^+_\phi$ if and only if it is carried by $\mathcal{A}_{na}(\Lambda^+_\phi)$.
  \item $\mathcal{A}_{na}(\Lambda^+_\phi)$ does not depend on the choice of the CT representing $\phi$.
  \item  Given $\phi, \phi^{-1} \in \out$ both rotationless elements and a dual lamination pair $\Lambda^\pm_\phi$ we have $\mathcal{A}_{na}(\Lambda^+_\phi)= \mathcal{A}_{na}(\Lambda^-_\phi)$
  \item $\mathcal{A}_{na}(\Lambda^+_\phi)$ is a free factor system if and only if the stratum $H_r$ is not geometric.
  \item $\mathcal{A}_{na}(\Lambda^+_\phi)$ is a malnormal subgroup system.

 \end{enumerate}

\end{lemma}

\subsection{Weak attraction theorem }
\label{sec:9}
\begin{lemma}[\cite{HM-13c} Corollary 2.17]
\label{WAT}
 Let $\phi\in \out$ be a rotationless and exponentially growing. Let $\Lambda^\pm_\phi$ be a dual lamination pair for $\phi$. Then for any line $\gamma\in\mathcal{B}$ not carried by $\mathcal{A}_{na}(\Lambda^{\pm}_\phi)$ at least one of the following hold:
\begin{enumerate}
 \item $\gamma$ is attracted to $\Lambda^+_\phi$ under iterations of $\phi$.
   \item $\gamma$ is attracted to $\Lambda^-_\phi$ under iterations of $\phi^{-1}$.
\end{enumerate}
Moreover, if $V^+_\phi$ and $V^-_\phi$ are attracting neighborhoods for the laminations $\Lambda^+_\phi$ and $\Lambda^-_\phi$ respectively, there exists an integer $l\geq0$ such that at least one of the following holds:
\begin{itemize}
 \item $\gamma\in V^-_\phi$.
\item $\phi^l(\gamma)\in V^+_\phi$
\item $\gamma$ is carried by $\mathcal{A}_{na}(\Lambda^{\pm}_\phi)$. 
\end{itemize}

\end{lemma}

 \subsection{Free-by-cyclic extensions for exponentially growing $\phi$}
 The method of proof followed in this work was developed by the author in \cite{self2}, where examples of free-by-free (strongly) relatively hyperbolic 
 extensions are constructed, which in turn was inspired by the work of Bestvina-Feighn-Handel in \cite{BFH-97}, where they construct free-by-free hyperbolic extensions.

 For the rest of the section we will assume that $\phi\in\out$ is an exponentially growing and rotationless outer automorphism. 
Let $\Lambda^\pm_\phi$ be a dual lamination pair associated to this automorphism.
Also let $f:G\to G$ be a CT map representing  $\phi$ and $H_r$ be the \textit{unique} exponentially growing strata associated to 
$\Lambda^+_\phi$ and $\mathcal{A}_{na}(\Lambda^+_\phi)$ be the nonattracting subgroup system of $\Lambda^+_\phi$. Recall that by 
construction, any conjugacy class is not weakly attracted to $\Lambda^+_\phi$ under iterates of $\phi$ if and only if it is carried by 
$\mathcal{A}_{na}(\Lambda^+_\phi$). Similarly let $f': G' \to G'$ be a CT representing $\phi^{-1}$ and $H'_s$ be the unique exponentially growing strata in $G'$ that is associated to
$\Lambda^-_\phi$.

We recall the notion of critical constant from \cite{BFH-97}.
 
 \textbf{Critical Constant: }
   Let $f:G\to G$ be a stable relative train track representative for some exponentially growing  
  $\phi\in\out$ with $H_r$ being an exponentially growing strata with associated Perron-Frobenius eigenvalue $\lambda$. 
  If $BCC(f)$ denotes the bounded cancellation constant for $f$, then the number $\frac{2BCC(f)}{\lambda-1}$ is called 
  the \textit{critical constant} for $f$. It can be easily seen that for every number $C>0$ that exceeds the 
  critical constant, there is some $\mu>0$ such that if $\alpha\cdot\beta\cdot\gamma$ is a concatenation of r-legal paths where 
   $\beta $ is some r-legal segment of length $\geq C$, then the r-legal leaf segment of 
   $f^k_\#(\alpha\cdot\beta\cdot\gamma)$ corresponding to $\beta$ has length $\geq \mu\lambda^k{|\beta|}_{H_r}$.
  To summarize, if we have a path in $G$ which has some r-legal ``central'' subsegment of length greater than the
    critical constant, then this segment is protected by the bounded cancellation lemma and under iteration, length of this segment grows exponentially.

 Following the work of Bestvina-Feighn-Handel \cite{BFH-97} we define the following notion of \textit{legality} for any number $C>0$ which exceeds 
 the critical constant for $f$.
 
 \textbf{Notation: } For any path $\alpha$ in $G$ let ${|\alpha|}_{\langle\mathcal{Z}, \hat{\sigma}\rangle}$ denote the edge length of $\alpha$ in $G$
 relative to ${\langle\mathcal{Z}, \hat{\sigma}\rangle}$, i.e. length of $\alpha$ in $G$, not counting 
 the copies of subpaths of $\alpha$ which are carried by ${\langle\mathcal{Z}, \hat{\sigma}\rangle}$, where $\sigma$ is the 
 unique indivisible Nielsen path of height $r$ (if it exists). In what follows $\rho$ will denote the 
 unique indivisible Nielsen path of height $s$ in the CT map $f':G'\to G'$ for $\phi^{-1}$. From train track theory we know that 
 the conjugacy classes of $\rho$ and $\sigma$ are same upto reversal, since the laminations associated to the strata $H_r$ and $H'_s$ are 
 dual to each other. Also recall from Bestvina-Feighn-Handel train track theory that $\sigma$ has exactly one illegal turn in $H_r$ and hence does not occur as a subpath of 
 any generic leaf of $\Lambda^+_\phi$.
 
 \begin{definition}
  For any circuit $\alpha$ in $G$, the $H_{r}$-legality of $\alpha$ is defined as the ratio 
  $${LEG}_{H_r^\sigma}(\alpha):= \frac{\text{sum of lengths of r-legal generic leaf segments of } \alpha \text{ of length} \geq C}{{|\alpha|}_{\langle\mathcal{Z}, \hat{\sigma}\rangle}}$$
 if ${|\alpha|}_{\langle\mathcal{Z}, \hat{\sigma}\rangle}\neq 0$. Otherwise, if ${|\alpha|}_{\langle\mathcal{Z}, \hat{\sigma}\rangle}=0$, define ${LEG}_{H_r^\sigma}(\alpha)=0$.
 \end{definition}

 For the rest of the paper, fix some $C$ greater than the critical constants of $f, f'$.
 
 Now we choose a long generic leaf segment $\gamma^+$ of some generic leaf of $\Lambda^+_\phi$, we may assume that 
 ${|\gamma^+|}_{\langle\mathcal{Z}, \hat{\sigma}\rangle}\gg C$, 
 to define a weak attracting neighborhood $V^+_\phi$ for $\Lambda^+_\phi$.
 Similarly choose a long generic leaf segment $\gamma^-$ such that ${|\gamma^-|}_{\langle\mathcal{Z}', \hat{\rho}\rangle}\gg C$ 
 and define an attracting neighborhood $V^-_\phi$. The weak attraction theorem tells us that for any conjugacy class 
 $\alpha$ that is not carried by $\mathcal{A}_{na}(\Lambda^\pm_\phi)$, there exists some $M>0$ such that 
 either $\alpha\in V^-_\phi$ or $\phi^m_\#(\alpha)\in V^+_\phi$ for every $m\geq M$. The following lemma shows that such an $M$ gives us enough legality 
  to eventually supersede  ${|\alpha|}_{\langle\mathcal{Z}, \hat{\sigma}\rangle}$.
  
  \begin{lemma}\label{legality}
 Suppose $\phi\in\out$ is rotationless and exponentially growing with a lamination pair $\Lambda^\pm_\phi$ which are topmost for 
 $\phi, \phi^{-1}$.
 Then there exists $\epsilon>0$ and some $M>0$ such that for every conjugacy class 
 $\alpha$ not carried by $\mathcal{A}_{na}(\Lambda^\pm_\phi)$, either ${LEG}_{H_r^\sigma}(\phi^m_\#(\alpha)) \geq \epsilon$ or 
 ${LEG}_{H_s^{'\rho}}(\phi^{-m}_\#(\alpha))\geq \epsilon$ for every $m\geq M$.
\end{lemma}

\begin{proof}
 Let $M$ be the constant from the weak attraction theory discussed above (right before the statement of the lemma). Suppose there does not 
 exist any such $\epsilon >0$ which satisfies the conclusion of our statement. We argue to a contradiction.
 
 Let $\{\alpha_i\}$ be a sequence of conjugacy classes, none of which are carried by $\mathcal{A}_{na}(\Lambda^\pm_\phi)$, such that both 
 ${LEG}_{H_r^\sigma}(\phi^M_\#(\alpha_i))\to 0$ and ${LEG}_{H^{'\rho}_s}(\phi^{-M}_\#(\alpha_i))\to 0$. Now arguing as in \cite[Lemma 4.6]{self2}, one can 
 use the bounded cancellation lemma to find 
  subpaths $\beta_i$ contained in $\alpha_i$ such that ${|\beta_i|}_{\langle\mathcal{Z}, \hat{\sigma}\rangle}\to\infty$ and 
  $\phi^M_\#(\beta_i)$ does not contain any legal generic leaf segment of 
 length $\geq C$ in $H_r$ except perhaps at the endpoints, i.e. $\phi^M_\#(\beta_i)$ does not have any central generic leaf segment of length $\geq C$ in $H_r$. Passing to a 
 subsequence if necessary, we may similarly assume  that ${|\beta_i|}_{\langle\mathcal{Z}', \hat{\rho}\rangle}\to\infty$ and 
 $\phi^{-M}_\#(\beta_i)$ does not contain any central generic leaf segment 
 (of $\Lambda^-_\phi$) of length $\geq C$ in $H'_s$. But $\beta_i$'s (by construction) are not carried by $\mathcal{A}_{na}(\Lambda^\pm_\phi)$ and have arbitrarily long 
 lengths in terms of paths in $H_r$ and $H'_s$ (since the laminations are topmost), not counting copies of paths carried by $\langle\mathcal{Z}, \hat{\sigma}\rangle$ and 
 $\langle\mathcal{Z}', \hat{\rho}\rangle$. After passing to a further subsequence if necessary 
 we may assume that there is common (sufficiently long) subpath $\tau$ which is crossed by all the $\beta_i$'s and 
 $\tau$ is not carried by $\langle\mathcal{Z}, \hat{\sigma}\rangle$ that is used to construct $\mathcal{A}_{na}(\Lambda^\pm_\phi)$. 
 This implies that some weak limit of some subsequence 
 of $\beta_i$'s is a line $l$ which contains the path $\tau$.
 
 Now observe that  $\phi^M_\#(l)$ (respc. $\phi^{-M}_\#(l)$) does not
 contain any subsegment of a generic leaf of $\Lambda^+_\phi$  (respc. $\Lambda^-_\phi$) of length $\geq C$ in $H_r$ (respc. $H'_s$). Using the 
 weak attraction theorem again, this implies that $l$ is carried by $\mathcal{A}_{na}(\Lambda^\pm_\phi)$ and hence contained in the nonattracting
 subgraph. But this contradicts that $l$ contains the subpath $\tau$.

 Hence such an $\epsilon> 0$ does indeed exist.
 
 \end{proof}

 Given a conjugacy class $\alpha$ not carried by $\mathcal{A}_{na}(\Lambda^\pm_\phi)$, the following lemma compares the growth of legal segments of $\alpha$ relative to the size of $\alpha$, 
 when both are  being measured in terms of paths not carried by $\langle\mathcal{Z}, \hat{\sigma}\rangle$, i.e. ${|\cdot|}_{\langle\mathcal{Z}, \hat{\sigma}\rangle}$.

 \begin{lemma}\label{flare}
 Suppose $\phi\in\out$ is rotationless and exponentially growing with a lamination pair $\Lambda^\pm_\phi$ which are topmost for 
 $\phi, \phi^{-1}$.
 Then for every $\epsilon>0$ and $A>0$, there is $M_1$ depending only on $\epsilon, A$ such that 
 if ${LEG}_{H_r^\sigma}(\alpha)\geq\epsilon $ for some circuit $\alpha$ then 
  \[{|f^m_\#(\alpha)|}_{\langle\mathcal{Z}, \hat{\sigma}\rangle}\geq A {|\alpha|}_{\langle\mathcal{Z}, \hat{\sigma}\rangle}\]  for every $m>M_1$. 
\end{lemma}
  
  \begin{proof}
   Since $\Lambda^+_\phi$ is a topmost lamination, any exponentially growing strata of height $> r$ is carried by $\langle\mathcal{Z}, \hat{\sigma}\rangle$. 
   So the only strata of height $s > r$ which can get attracted to $\Lambda^+_\phi$ are superlinear NEG edges of the form $E_s\mapsto E_su_s$, 
   where $u_s$ is a circuit contained in $G_{s-1}$ and gets attracted to $\Lambda^+_\phi$. Hence neither $E_s$ nor $u_s$ are carried by $\langle\mathcal{Z}, \hat{\sigma}\rangle$.

   Using the description of critical constant we can see that 
   \begin{align}
   {|f^m_\#(\alpha)|}_{\langle\mathcal{Z}, \hat{\sigma}\rangle} & \geq \mu\lambda^m\{\text{sum of r-legal generic leaf segments of $\alpha$ having length } \geq C\}\nonumber \\
      & \geq \mu\lambda^m\epsilon{|\alpha|}_{\langle\mathcal{Z}, \hat{\sigma}\rangle}\nonumber
   \end{align}
   
   Choose $m$ to be large enough so that $\mu\lambda^m\epsilon\geq A$.
   
  \end{proof}

  Recall the nonattracting subgroup system $\mathcal{A}_{na}(\Lambda^\pm_\phi) = \{[F^1], [F^2],... ,[F^k]\}$ is a mutually malnormal 
  subgroup system of finitely generated subgroups of $\F$ (hence quasiconvex). Hence one can form the coned-off space 
  $\widehat{\F}$ with respect to the collection $\{F^i\}$ and form a electrocuted metric space denoted by 
  ($\widehat{\F}, {|\cdot|}_{el}$). The group $\F$ is then strongly hyperbolic 
  relative to the collection $\{F^i\}$. Notice that this property of $\widehat{\F}$ is still true if we replace any of the 
  $F^i$'s with a conjugate of itself. Hence we may choose any representative (abusing the notation) $F^i$ for each $[F^i]$ and a lift $\Phi_i$ , such
  that $\Phi_i(F^i)=F^i$ and assume that $\F$ is (strongly) relatively hyperbolic with respect to this collection of subgroups.
  We stress the fact here that $\mathcal{A}_{na}(\Lambda^\pm_\phi)$ is not necessarily a free factor system.
  By ${||\alpha||}_{el}$ we denote the length of the shortest 
  representative of the conjugacy class $\alpha$ in this electrocuted metric and by ${|\alpha|}_{\langle\mathcal{Z}, \hat{\sigma}\rangle}$ we denote the
   length (relative to $\langle\mathcal{Z}, \hat{\sigma}\rangle$) of the circuit in $G$ that realizes this conjugacy class.

  The following lemma establishes that these two measurements are comparable.
   
   \begin{lemma}\label{comparison}
    Suppose $\phi\in\out$ is rotationless and exponentially growing with a lamination pair $\Lambda^\pm_\phi$ which are topmost for 
 $\phi, \phi^{-1}$ and $\alpha$ is any circuit in $G$ with ${LEG}_{H_r^\sigma}(\alpha)>0$. Then 
 there exists some $K>0$, independent of $\alpha$, such that $K \geq {|\alpha|}_{\langle\mathcal{Z}, \hat{\sigma}\rangle}/{||\alpha||}_{el}\geq \frac{1}{K}$.
   \end{lemma}

\begin{proof}
 Let $\tilde{G}$ denote the universal cover of the marked graph $G$. Consider $\tilde{G}$ equipped with the electrocuted metric obtained by electrocuting 
 images of all cosets of $F^i$'s, by using the marking on $\tilde{G}$, which is obtained by lifting the marking on $G$. 
 This electrocuted metric is denoted by ${|\cdot|}^{\tilde{G}}_{el}$. 
 Then any path in $\tilde{G}$ that is entirely contained in the 
 copy of a coset of some $F^i$ projects down to a path in $G$ which is carried by $\langle\mathcal{Z}, \hat{\sigma}\rangle$. Note that the 
 assumption on legality of $\alpha$ ensures that we are working with a conjugacy class which is not carried by $\mathcal{A}_{na}(\Lambda^\pm_\phi)$, 
 equivalently the circuit $\alpha$ representing such a conjugacy class is not carried by $\langle\mathcal{Z}, \hat{\sigma}\rangle$, hence 
 ${|\alpha|}_{\langle\mathcal{Z}, \hat{\sigma}\rangle}\geq 1$.
 
 Let $E_i$ be some edge not carried by $\langle\mathcal{Z}, \hat{\sigma}\rangle$. Then the lift $\tilde{E}_i$ is a path in $\tilde{G}$ which 
 is not entirely contained in the copy of some coset of $F^i$. Let $L_2 = \text{max}\{{|\tilde{E}_i|}^{\tilde{G}}_{el}\}$ where $E_i$ varies 
 over all edges of $G$ which are not carried by $\langle\mathcal{Z}, \hat{\sigma}\rangle$. If $\alpha$ is any circuit in $G$ which is not carried 
 by the nonattracting subgroup system, then  
 \[ {|\tilde{\alpha}|}^{\tilde{G}}_{el}\leq L_2 {|\alpha|}_{\langle\mathcal{Z}, \hat{\sigma}\rangle} 
 \implies {|\tilde{\alpha}|}^{\tilde{G}}_{el}/ {|\alpha|}_{\langle\mathcal{Z}, \hat{\sigma}\rangle}\leq L_2\]
 
 Next suppose that $\tilde{\beta}_i$ is some geodesic path in $\tilde{G}$ (in the electrocuted metric) which connects copies of two electrocuted cosets and 
 does not intersect any copy of any electrocuted coset except at the endpoints. Note that there are only finitely many such paths upto translation
 in $\tilde{G}$. Let $L'_1=\text{max}\{{|\beta_i|}_{\langle\mathcal{Z}, \hat{\sigma}\rangle}\}$, where $\beta_i$ is the projection of 
 $\tilde{\beta}_i$ to $G$ followed by tightening. Also consider all $\tilde{\alpha}^j_i$'s where $\tilde{\alpha}^j_i$ varies over geodesic paths   
  inside the copy of the identity coset of $F^j$ in $\tilde{G}$ representing a generator $g^j_i$ of $F^j$, under standard metric on $\tilde{G}$ 
  (i.e. we are recording the length of geodesic paths representing a generator of $F^j$ with the standard metric on $\tilde{G}$).
 Let $L''_1=\max\{{|\alpha_i^j|}_G\}$,  where $\alpha^j_i$ is the projection of $\tilde{\alpha}^j_i$ to $G$ followed by tightening (note that 
 the measurement done for $L_1''$ is in terms of standard path metric on $G$).
  
  Suppose $w\in\F$ is some cyclically reduced word not in the union of $F^i$'s. Let $\tilde{\alpha}$ be a path in $\tilde{G}$ that 
  represents the geodesic connecting the identity element to $w$, under the lift of the marking on $G$. Suppose $\tilde{\alpha}= u_1 X_1 u_2 X_2...u_s X_s$, where $X_i$'s are 
  geodesic paths in $\tilde{G}$ connecting two points in a copy of some coset via the attached cone-point and $u_i$'s are geodesic paths in 
  $\tilde{G}$ which connect copies of two electrocuted cosets and does not pass through any cone-point. Note that $u_i$'s are concatenation 
  of paths of type $\tilde{\beta}_i$'s described above.
  Under this setup we have 
  \[{|\tilde{\alpha}|}^{\tilde{G}}_{el}={|u_1|}^{\tilde{G}}_{el} + {|u_2|}^{\tilde{G}}_{el} + ...+ {|u_s|}^{\tilde{G}}_{el} + s \]
 Modify $\tilde{\alpha}$ by replacing each $X_i$ with a (minimal) concatenation of translations of paths of type $\tilde{\alpha}^j_i$ in $\tilde{G}$
 described above, and look at the 
  projection of this modified path obtained from $\tilde{\alpha}$ to $G$ followed by tightening. Denote the tightened, projected path by $\alpha$. Then 
  $\alpha$ is a path in $G$ such that after accounting for cancellations that appear in the modification of $\tilde{\alpha}$, we have
 \[{|\alpha|}_{\langle\mathcal{Z}, \hat{\sigma}\rangle}\leq {|\tilde{\alpha}|}^{\tilde{G}}_{el}L_1 + 2sL_1
 \implies \frac{{|\alpha|}_{\langle\mathcal{Z}, \hat{\sigma}\rangle}}{{|\tilde{\alpha}|}^{\tilde{G}}_{el}}
 \leq L_1+\frac{2sL_1}{{|\tilde{\alpha}|}^{\tilde{G}}_{el}}\leq 3L_1\]
 
 The other possibilities in the presentation of $\tilde{\alpha}$ (in terms of $u_i$'s and $X_i$'s) are handled similarly.
 Hence we have
 \[\frac{1}{3L_1}\leq  {|\tilde{\alpha}|}^{\tilde{G}}_{el}/{|\alpha|}_{\langle\mathcal{Z}, \hat{\sigma}\rangle} \leq L_2 \addtag \]

 Now using the lift of marking map on $G$ to $\tilde{G}$ one can show by similar arguments as above that there exists some $K'>0$ such that 
 for every cyclically reduced word $w\in\F\setminus\cup F^i$ we have
 \[\frac{1}{K'}\leq {|w|}_{el}/{|\tilde{\alpha}_w|}^{\tilde{G}}_{el} \leq K' \addtag \] where $\tilde{\alpha}_w$ is the electrocuted geodesic in 
 $\tilde{G}$ connecting the image of identity element to image of $w$ under the marking map on $\tilde{G}$ and ${|w|}_{el}$ is the length of the 
 electrocuted geodesic in $\hat{\F}$ connecting identity element and $w$. Hence combining the inequalities (1) and (2) 
 above we can conclude that there 
 exists some $K>0$ such that \[ K \geq {|\alpha|}_{\langle\mathcal{Z}, \hat{\sigma}\rangle}/{||\alpha||}_{el}\geq \frac{1}{K} \]
 \end{proof}

\begin{proposition}[Conjugacy flaring]
 \label{conjflare}
 Suppose $\phi\in\out$ is rotationless and exponentially growing with a lamination pair $\Lambda^\pm_\phi$ which are topmost.
  Then there exists some $M_0>0$ such that for every conjugacy class $\alpha$ not carried by $\mathcal{A}_{na}(\Lambda^\pm_\phi)$, we have 
  $$3{||\alpha||}_{el} \leq \mathsf{max} \{{||\phi^M_\#(\alpha)||}_{el}, {||\phi^{-M}_\#(\alpha)||}_{el}\} $$
 for every $M\geq M_0$.
 \end{proposition}

 \begin{proof}
  Let $M$ be as in Lemma \ref{legality} and choose some number $D>0$ such that we have 
  ${|\phi^{M}_\#(\alpha)|}_{\langle\mathcal{Z}, \hat{\sigma}\rangle} \geq {|\alpha|}_{\langle\mathcal{Z}, \hat{\sigma}\rangle}/D$ and 
  ${|\phi^{-M}_\#(\alpha)|}_{\langle\mathcal{Z}', \hat{\rho}\rangle} \geq {|\alpha|}_{\langle\mathcal{Z}', \hat{\rho}\rangle}/D$.  
  By applying Lemma \ref{legality} again there exists some $\epsilon >0$ such that 
  either ${LEG}_{H_r^\sigma}(\phi^M_\#(\alpha)) \geq \epsilon$ or ${LEG}_{H^{'\rho}_s}(\phi^{-M}_\#(\alpha)) \geq \epsilon$.
  
  By applying Lemma \ref{comparison} we may choose some constant $K$ such that for every conjugacy class $\alpha$ 
  as above, we have   either 
  $K \geq {|\phi^M_\#(\alpha)|}_{\langle\mathcal{Z}, \hat{\sigma}\rangle}/{||\phi^M_\#(\alpha)||}_{el}\geq \frac{1}{K}$ or 
  $K\geq{|\phi^{-M}_\#(\alpha)|}_{\langle\mathcal{Z}', \hat{\rho}\rangle}/{||\phi^{-M}_\#(\alpha)||}_{el} \geq \frac{1}{K}$.

  For concreteness assume that ${LEG}_{H_r^\sigma}(\phi^m_\#(\alpha)) \geq \epsilon$. Then by applying 
  Lemma \ref{flare} with $\epsilon$ and $A= 3DK^2$ we get that there exists some $M_1$ such that for all 
  $m\geq M_0:=\mathsf{max}\{M, M_1\}$ we have 
  \begin{equation}
  \begin{split}
  {||\phi^m_\#(\alpha)||}_{el} & \geq  \frac{1}{K}{|\phi^m_\#(\alpha)|}_{\langle\mathcal{Z}, \hat{\sigma}\rangle} \\
  & \geq \frac{1}{K}3DK^2{|\phi^{M_1}_\#|}_{\langle\mathcal{Z}, \hat{\sigma}\rangle} \\
   & \geq 3DK \frac{1}{D}{|\alpha|}_{\langle\mathcal{Z}, \hat{\sigma}\rangle}= 3K{|\alpha|}_{\langle\mathcal{Z}, \hat{\sigma}\rangle}\\
   & \geq 3K\frac{1}{K}{||\alpha||}_{el} = 3{||\alpha||}_{el}
  \end{split}
 \end{equation}
  
 \end{proof}

 The following lemma proves that conjugacy flaring implies the Mj-Reeves \textit{cone-bounded hallways strictly flaring} condition. The technique used in the proof 
 is a generalization of the one used by Bestvina-Feighn-Handel in \cite[Theorem 5.1]{BFH-97}. Recall that we are working with 
 an exponentially growing outer automorphism with a dual lamination pair $\Lambda^\pm_\phi$ such that 
 $\mathcal{A}_{na}(\Lambda^\pm_\phi) = \{[F^1], [F^2],... ,[F^k]\}$.
 
 \begin{proposition}[Strictly flaring]
 \label{strictflare}
 Suppose $\phi\in\out$ is rotationless and exponentially growing with a lamination pair $\Lambda^\pm_\phi$ which are topmost.
  There exists some $N>0$ such that for every word $w\in \F\setminus \cup F^i$ and any lift $\Phi$ of $\phi$, we have 
  $$2{|w|}_{el} \leq \mathsf{max}\{{|\Phi^n_\#(w)|}_{el}, {|\Phi^{-n}_\#(w)|}_{el}\}$$
  for every $n \geq N$.
 \end{proposition}

 \begin{proof}
 Recall that for any subgraph $H\subset G$ we can use the fundamental groups of the noncontractible
 components of $H$ to form a free factor system. 
  First observe that from the construction of $\mathcal{A}_{na}(\Lambda^\pm_\phi)$ we know that the strata $H_r$ is not a part 
  of the nonattracting subgraph $Z$ that is used to define $\mathcal{A}_{na}(\Lambda^\pm_\phi)$. Hence we can always find a free factor 
  system $\mathcal{F}$, defined by the noncontractible components of $H_r$, with one component $[K^i]$ for each noncontractible component
  of $H_r$. Without loss, assume that $K^1 * K^2 * ... * K^p * B = \F$, where $B$ is perhaps trivial. 
  
  Let $L = \mathsf{max}\{{|\Phi^i_\#(k_j)|}_{el}\arrowvert i=0,\pm1,\pm2,...,\pm M_0\}$ where 
  $k_j\in K^j$ varies over all the basis elements of $K^j$ for some chosen basis of $\F$,  and $M_0$ is the 
  constant from Lemma \ref{conjflare}.  By description of $L$ we may assume that ${|\Phi^{i}_\#(k_j)|}_{el}\geq 1/L$ 
  for all $i, j$ as described above. 
  
  \textbf{Case 1: } Assume $w\in\F\setminus \cup F^i$ and ${|w|}_{el} \geq L-3$.

  The proof is by induction. For the base case let $n = M_0$.

   If $w$ is a cyclically reduced word then conjugacy class of $w$ is not carried by $\mathcal{A}_{na}(\Lambda^\pm_\phi)$ and 
  so by using Lemma \ref{conjflare} we have  
  $$ \mathsf{max}\{{|\Phi^n_\#(w)|}_{el}, {|\Phi^{-n}(w)|}_{el}\} \geq 3{|w|}_{el} \geq 2{|w|}_{el} $$
  If $w$ is not cyclically reduced then we can choose a basis element $k\in K^j$, for some j, such that 
  $kw\in \F\setminus \cup F^i$  is a cyclically reduced word.
  Hence we get the same inequality as above, but with $w$ being substituted by $kw$.
  
  For sake of concreteness suppose that ${|\Phi^n_\#(kw)|}_{el} \geq 3{|kw|}_{el}$. Then we have 
  $3{|kw|}_{el}\leq {|\Phi^n_\#(kw)|}_{el}\leq  {|\Phi^n_\#(w)|}_{el}+{|\Phi^n_\#(k)|}_{el}\leq {|\Phi^n_\#(w)|}_{el} +L $.
  This implies that $3 + 3{|w|}_{el} - L \leq 3{|kw|}_{el} - L \leq {|\Phi^n_\#(w)|}_{el}$ 
  since ${|k|}_{el} = 1$ (because $k$ is a basis element) and there is no cancellation between $k$ and $w$. 
  Since we have  ${|w|}_{el} \geq L-3$, the above inequality then implies $2{|w|}_{el}\leq {|\Phi^n_\#(w)|}_{el}$
  and we are done with the base case for our inductive argument.
  
  Now assume that $M_0< n $ for the inductive step. First observe that from what we have proven so far,
  given any integer $s>0$ we have either ${|\Phi^{sM_0}_\#(w)|}_{el}\geq 2^{s}{|w|}_{el}$ or
  ${|\Phi^{-sM_0}_\#(w)|}_{el} \geq 2^{s}{|w|}_{el}$. Fix some positive integer $s_0$ such that 
  $2^{s_0}> 2L$. Any integer $n>s_0M_0$ can be written as $n= sM_0 + t$ where $0\leq t <M_0$ and 
  $s_0\leq s$. If ${|\Phi^{sM_0}_\#(w)|}_{el}\geq 2^{s}{|w|}_{el}$ then we can deduce 
  $$ {|\Phi^n_\#(w)|}_{el} = {|\Phi^{sM_0+t}_\#(w)|}_{el}\geq 2^s{|w|}_{el}/L \geq 2{|w|}_{el}$$ 
  Similarly 
  when ${|\Phi^{-sM_0}_\#(w)|}_{el}\geq 2^{s}{|w|}_{el}$ one proves by a symmetric argument that 
  ${|\Phi^{-n}_\#(w)|}_{el} \geq 2{|w|}_{el}$.
  
  \textbf{Case 2: } Assume $w\in\F\setminus \cup F^i$ and ${|w|}_{el} < L-3$.
  
  Firstly we note that $w\notin \cup F^i$ implies that $0 < {|w|}_{el}$. If $w$ is not conjugate to some element of 
  $F^i$, then we can argue as in the beginning of Case 1 to show that that either 
  ${|\Phi^n_\#(w)|}_{el}\geq 2{|w|}_{el}$ or ${|\Phi^{-n}_\#(w)|}_{el}\geq 2{|w|}_{el}$ for all $n\geq M_0$.
  
  If $w$ is conjugate to some element of $F^i$ then we can write $w=ugu^{-1}$ for some generator
  $u\in K^j$ such that $u$ is not conjugate to any word in any of the $F^j$'s and $g\in F^i$. ${|w|}_{el} < L-3$ implies that  
  $${|w|}_{el} \leq {|u|}_{el} + {|u^{-1}|}_{el} = 2 {|u|}_{el} < L-3$$
  Now observe that under iteration of $\Phi$, the reduced word $g$ has polynomial growth in the electrocuted metric 
  since it's conjugacy class is carried by the nonattracting subgroup system $\mathcal{F}$, whereas the word 
  $u$ grows exponentially under iteration of $\phi$. Hence we can conclude that 
  ${|\Phi^s_\#(w)|}_{el} \geq {|\Phi^s_\#(u)|}_{el}$ for all $s>0$ and thus $w$ has exponential growth in the electrocuted 
  metric. Now choose some $N_w$ such that ${|\Phi^{N_w}_\#(u)|}_{el}\geq 4 {|u|_{el}}\geq 2{|w|}_{el}$.  
  Observe that the bounded cancellation lemma tells us that $N_w$'s obtained from this subcase depend only on the conjugating word $u$ and not on 
  $g$. Hence they are only finitely many $N_w$'s.
  
  Finally we let $N$ be max of $M_0$ from Case 1 and all the $N_w$'s from case 2 
  and we have the desired conclusion.

 \end{proof}

 For convenience of the reader we reproduce some of the terminologies used in the Mj-Reeves strong 
 combination theorem for relative hyperbolicity and put it in the context that we have here.
 We perform a partial electrocution of the extension group $\Gamma$ $$1\to\F\to\Gamma\to\langle\phi\rangle\to 1 $$
  with respect the collection of 
  subgroups $\{F^i\}$ and denoted it by $\widehat{\Gamma}$. We also perform an electrocution of 
  $\F$ with respect to the collection $\{F^i\}$, denoted by $(\widehat{\F}, {|\cdot|}_{el})$, and since $\mathcal{A}_{na}(\Lambda^\pm_\phi)$
  is a malnormal subgroup system, $\F$ is (strongly) relatively hyperbolic with respect to the collection $\{F^i\}$. 
  The Cayley graph of the quotient group $\langle\phi\rangle$ (denoted by $T$) 
  being a tree (in fact a line here), gives us a tree of (strongly) relatively hyperbolic spaces with vertex spaces being identified 
  with cosets of $\F$. Thus we may regard the Cayley graph of $\Gamma$ as the tree $T$ of (strongly) relatively hyperbolic 
  spaces and then $\widehat{\Gamma}$ is the induced tree of coned-off spaces

 Next step is to choose lifts $\Phi_i$ of $\phi$ such that $\Phi_i(F^i)=F^i$. This gives us an induced short exact sequence 
 $$1\to F^i\to\Gamma_i\to\langle\phi_i\rangle \to 1$$, where $\phi_i\in\mathrm{Out}(F^i)$.
 
 \begin{description}
 \item \textbf{Induced tree of coned-off spaces}: $\hat{\Gamma}$ as described above is the induced 
 tree of coned-off spaces. This is the first stage of electrocution (called \textit{partial electrocution} in the 
 Mj-Reeves paper). (see \cite[Definition 3.4]{MjR-08} and the discussion that follows after it.) 
 
 In our case the edge spaces and vertex spaces are identified with cosets of $\F$ and the maps from the edge space to the vertex (the endpoints of the concerned edge) spaces 
 are in fact quasi-isometries. Since $\phi$ reserves each $[F^i]$ for each $i$, 
  is fairly straightforward to check that in our case $\Gamma$, viewed as a tree of (strongly) relatively hyperbolic spaces, satisfies the first three conditions 
 of the \cite[Theorem 4.6]{MjR-08}, namely the \textit{q.i.-embedded condition, strictly type preserving condition, the q.i.-preserving 
 electrocution condition}.
 
 \item \textbf{Cone-locus} (from \cite{MjR-08}): The cone locus of $\hat{\Gamma}$, induced tree of coned-off spaces, is the forest whose vertex set consists of the cone-points of the 
vertex spaces of $\hat{\Gamma}$ and whose edge set consists of the cone-points in the edge spaces of $\hat{\Gamma}$. The incidence relations 
of the cone locus is dictated by the incidence relations in $T$.

Connected components of cone-locus can be identified with subtrees of $T$. 
Each connected component of the cone-locus is called a \textit{maximal cone-subtree}. 
The collection of maximal cone-subtrees is denoted by $\mathcal{T}$ and each element of 
$\mathcal{T}$ is denoted by $T_j$. Each $T_j$ gives rise to a tree $T_j$ of horosphere-like subsets depending on 
which cone-points arise as vertices and edges of $T_j$. The metric space that $T_j$ gives rise 
to is denoted by $C_j$ and is called \textit{maximal cone-subtree of horosphere-like
spaces}. The collection of $C_j$ is denoted by $\mathcal{C}$.

In our context, the collection $\mathcal{C}$ corresponds to the collection of 
cosets of $\Gamma_i$ as a subgroup of $\Gamma$, for each $i$. Note that the partially electrocuted metric space, $\hat{\Gamma}$,
can be viewed as electrocuting cosets of $F^i$ in $\Gamma_i$ across all 
cosets of $\Gamma_i$ in $\Gamma$, for each $i$. The maximal cone-subtrees, $T_j$'s, are obtained from
this electrocution of $\Gamma_i$; each coset of $\Gamma_i$, after electrocution of 
cosets of $F^i$ inside $\Gamma_i$, gives us a maximal cone-subtree, i.e. the $C_j$'s are electrocuted to $T_j$'s in the first 
step of electrocution.

  \item \textbf{Hallway}(from \cite{BF-92}): A disk $f: [-m, m]\times I\to \hat{\Gamma}$ is a hallway 
  of length $2m$ if it satisfies the following conditions:
  \begin{enumerate}
   \item $f^{-1}(\cup \hat{\Gamma}_v: v\in T) = \{-m, ...., m\}\times I$.
   \item $f$ maps $i\times I$ to a geodesic in $\hat{\Gamma}_v$ for some vertex space.
   \item $f$ is transverse, relative to condition (1) to $\cup \hat{\Gamma}_e$.
  \end{enumerate}
  
  Recall that in our case, the vertex spaces being considered above are just copies of $\hat{\F}$ with the electrocuted metric (obtained 
  from $\F$ by coning-off the collection of subgroups $F^i$).

\item \textbf{Thin hallway}: A hallway is $\delta-$thin if $d(f(i,t), f(i+1,t)) \leq \delta$ for 
all $i, t$.
\item A hallway is $\lambda-$hyperbolic if 
$$\lambda l(f(\{0\}\times I)) \leq \text{max}\{l(f(\{-m\}\times I)), l(f(\{m\}\times I))\}$$
 
 \item \textbf{Essential hallway}: A hallway is essential if the edge path in $T$ resulting from projecting $\hat{\Gamma}$ onto $T$
 does not backtrack (and hence is a geodesic segment in the tree $T$).
 \item \textbf{Cone-bounded hallway} (from \cite[Definition 3.4]{MjR-08}): An essential hallway of length $2m$ is cone-bounded if 
 $f(i\times \partial I)$ lies in the cone-locus for $i=\{-m, ...., m\}$. 
 
 Recall that in our case, the connected components of the cone-locus are $T_j$'s which are the cosets of 
 $\Gamma_i$ (post electrocuting the cosets of $F^i$ in $\Gamma_i$) inside $\Gamma$.

 \item \textbf{Hallways flare condition} (from \cite{BF-92}, \cite{MjR-08}): 
 The induced tree of coned-off spaces, $\hat{\Gamma}$, is said to 
 satisfy the hallways flare condition if there exists $\lambda> 1, m\geq 1$ such that 
 for all $\delta$ there is some constant $C(\delta)$ such that any $\delta-$thin essential 
 hallway of length $2m$ and girth at least $C(\delta)$ is $\lambda-$hyperbolic.

 In our context, Proposition \ref{strictflare} establishes that hallways flare condition is satisfied for $\hat{\Gamma}$ (with $\lambda= 2$), 
 since $\hat{\Gamma}$ is obtained from $\Gamma$ by electrocuting cosets of $F^i$, for each $i$.
 Thus $\hat{\Gamma}$ is a hyperbolic metric space by using the Bestvina-Feighn combination theorem and $\hat{\Gamma}$ is weakly hyperbolic 
 relative to the collection $\mathcal{T}$ (\cite[Lemma 3.8]{MjR-08}).
 
 Once the hyperbolicity of $\hat{\Gamma}$ is established, we can proceed to the second stage of 
 electrocution, where we electrocute the the maximal cone-subtrees, $T_j$'s. The resulting space is quasi-isometric to electrocuting 
 the $C_j$'s inside $\Gamma$ to a cone-point directly (see proof of \cite[Theorem 4.1]{MjR-08}) and thus this step shows that $\Gamma$ 
 is weakly hyperbolic relative to the collection of spaces $C_j$ (equivalently, relative to the collection of mapping tori subgroups $\Gamma_i$).
 
 \item \textbf{Cone-bounded hallways strictly flare condition} (from \cite[Definition 3.6]{MjR-08}): 
 The induced tree of coned-off spaces $\hat{\Gamma}$, is said to satisfy the cone-bounded hallways strictly flare condition if there exists 
 $\lambda>1, m\geq 1$ such that any cone-bounded hallway of length $2m$ is $\lambda-$hyperbolic.

 In our case, this condition is also verified in Proposition \ref{strictflare}, since each connected component of the cone-locus is just $\Gamma_i$ with the 
 electrocuted metric obtained by coning-off the subgroup $F^i$ and it's cosets (in $\Gamma_i$) and this electrocuted $\Gamma_i$ can be viewed as a subspace  of 
 $\hat{\Gamma}$ (for which the flaring condition holds). 
 
 \end{description}

 Hence we have the following theorem by applying \cite[Theorem 4.6]{MjR-08}
 
 \begin{proposition}\label{relhypext}
  Let $\phi\in\out$ be rotationless and exponentially growing outer automorphism equipped with a dual lamination pair $\Lambda^\pm_\phi$, which 
  are topmost.
  Also let $\mathcal{A}_{na}(\Lambda^\pm_\phi) = \{[F^1], [F^2],...., [F^k]\} $ denote the nonattracting subgroup system for $\Lambda^\pm_\phi$.
  If  $F^i$'s denote representatives of $[F^i]$ such that $\Phi_i(F^i)=F^i$ for some lift $\Phi_i$, then the extension group 
  $\Gamma$ in the short exact sequence $$1 \to \F \to \Gamma \to \langle \phi \rangle \to 1 $$
  is strongly hyperbolic relative to the collection of subgroups $\{F^i \rtimes_{\Phi_i} \mathbb{Z}\}$.
 \end{proposition}

 We would like to point out that one could have also used the Combination theorem due to Gautero \cite{Gau-16} 
 to deduce the same result, since Proposition \ref{conjflare} essentially shows that the automorphism $\phi$ is ``relatively hyperbolic'' 
 in the sense of Gautero-Lustig \cite{Gau-07}.
 
 \subsection{Atoroidal case}
 
 Observe that if $\phi$ is \textit{atoroidal} (i.e. does not have any periodic conjugacy class) then $\phi$ is necessarily 
 exponentially growing (see \cite[Lemma 3.1]{self17}). This can also be seen by using the Bestvina-Feighn-Handel theorem developed in 
  \cite{BFH-05} which implies that any polynomially growing outer automorphism necessarily has a periodic conjugacy class.
 Thus we can use Proposition \ref{relhypext} to set up a recursion 
 on each mapping tori $\{F^i \rtimes_{\Phi_i}\mathbb{Z}\}$ in the following way:
 \begin{enumerate}
  \item If the mapping tori $\{F^i \rtimes_{\Phi_i}\mathbb{Z}\}$ is hyperbolic, we stop. 
  \item If the restriction of $\phi$ to  $F^j$ is polynomially growing then using the Kolchin-type theorem of Bestvina-Feighn-Handel 
  \cite{BFH-05} we know that $\phi$ must fix some conjugacy class in $F^j$ and therefore $\phi$ cannot be atoroidal. Hence this case is 
  not possible.
  \item If $\text{rank}(F^i)\leq 2$ , then again this implies $\phi$ fixes some conjugacy class and therefore cannot be atoroidal. Hence 
  this case is not possible.
  \item If none of the above cases happen, then we may continue recursively by considering the restriction $\phi\in\mathsf{Out}(F^i)$ (which is also 
  rotationless)
  and applying proposition \ref{relhypext}.
  Since each subsequent component, say $F^{i,j}_\Phi$, obtained from $F^i_\Phi$ in this process is a proper free factor of $\F$ by itself (due to item (5) in 
  \ref{NAS}), the rank of such $F^{i,j}_\Phi$ drops at each step of recursion. 
  Therefore this process must stop when we have $\mathsf{rank}(F^{i,j,k,..,s})=3$
  and  the mapping tori corresponding to this invariant subgroup must be hyperbolic (any outer automorphism of a free group of rank 2 always fixes 
  some conjugacy class).
    
 \end{enumerate}
 
 This shows that for an atoroidal $\phi$, the collection of parabolic subgroups obtained in Proposition \ref{relhypext} can be made finer
 (in finitely many steps) until eventually we have that 
 each subgroup in the collection is a hyperbolic group. Now using two well known facts in the theory of hyperbolic groups, namely:
 \begin{itemize}
  \item If $\Gamma$ is (strongly) hyperbolic relative to the finite collection ${\{H_i\}}_i$ and each $H_i$ is (strongly) hyperbolic relative to 
  the finite collection ${\{K^i_j\}}_j$, then $\Gamma$ is (strongly) hyperbolic relative to the finite collection ${\{K^i_j\}}_{i,j}$.
  \item If $\Gamma$ is (strongly) hyperbolic relative to the finite collection ${\{H_i\}}_i$ and each $H_i$ is hyperbolic, then $\Gamma$ is hyperbolic.
 \end{itemize}
 we can conclude that $\Gamma$ itself must be hyperbolic when $\phi$ is atoroidal. Thus, as a corollary of Proposition \ref{relhypext}, we have obtained a 
 new proof of Brinkmann's theorem:
 
 \begin{corollary}\label{atoroidal}
  $\phi\in\out$ is atoroidal if and only if the extension group $\Gamma$, in the short exact sequence $1\to\F\to\Gamma\to\langle\phi\rangle\to 1 $, 
  is a hyperbolic group.
   \end{corollary}
   
   \begin{proof}
    Pass to a rotationless power $\phi'$ of $\phi$ and consider the mapping torus $\Gamma'$ of $\phi'$. The discussion above shows that 
    $\Gamma'$ is hyperbolic. Since $\Gamma'$ is a finite-index subgroup of the mapping torus, $\Gamma$, of $\phi$, they are quasi-isometric. 
    Therefore $\Gamma$ is also hyperbolic. 
   \end{proof}

 The converse of Proposition \ref{relhypext} has already been proven, and is present implicitly in the work of Macura \cite{Mak-02} as was first 
 pointed out in the work of Hagen-Wise \cite{HagenW-16}.  In \cite{Mak-02} the author studies the divergence function for the free-by-cyclic 
 extensions induced by polynomially growing outer automorphisms and proves that the divergence function for an extension induced by 
 an outer automorphism of growth order $r$, is approximately $x^{r+1}$. But it is well known that relatively hyperbolic groups have 
 exponential divergence function (see \cite{Sisto-12}). The author thanks Mark Hagen for pointing this out. 
 
 \textbf{Remark:} In the work of Macura, it is 
 implicit that for a polynomially growing outer automorphism $\phi$, the induced mapping torus 
 group $\F\rtimes_\Phi\mathbb{Z}$ is a \textit{thick metric space} of order $r$ (in the sense of \cite{BDM-09})  
 for any lift $\Phi$ of $\phi$ and hence $\F\rtimes_\Phi\mathbb{Z}$ cannot be relatively hyperbolic with respect to any finite collection of 
 subgroups. It has been communicated to the author that Hagen and Niblo are  writing down a explicit proof of this fact \cite{HagNib-18}.
 
 Thus in light of the above discussion, we have the following theorem:
 \begin{theorem}\label{main}
  If $\phi\in\out$, then the extension group $\Gamma$ in the short exact sequence 
  $$1 \to \F \to \Gamma \to \langle \phi \rangle \to 1 $$ is relatively hyperbolic if and only if $\phi$ is exponentially growing.
 \end{theorem}
\begin{proof}
 Pass to rotationless power $\phi'$ of $\phi$ and consider the mapping torus $\Gamma'$ of $\phi'$. Proposition \ref{relhypext} shows that 
 $\Gamma'$ is (strongly) relatively hyperbolic group. Since $\Gamma'$ is quasi-isometric to $\Gamma$, by using the fact that 
 relative hyperbolicity is quasi-isometry invariant (\cite{Dr-09}), we can conclude that $\Gamma$ is also (strongly) relatively hyperbolic. 
 
\end{proof}

 Before we move on to the polynomially growing case
 we would like to remark that a somewhat similar looking result as in Proposition \ref{relhypext} 
 was claimed in an yet unpublished paper of Gautero-Lustig \cite{Gau-07}.
 What they show is that the mapping torus of any outer automorphism is (strongly) hyperbolic relative to the collection of 
 ``canonical'' subgroups which contain all the polynomially growing conjugacy classes. A result similar to the Gautero-Lustig theorem 
 can be easily deduced by a repeated application of 
 Proposition \ref{relhypext} on the peripheral subgroups given in the conclusion of that proposition (similar to the argument we gave for 
 the atoroidal case \ref{atoroidal}). The two facts that allows us to do a inductive argument in this case are the following
 \begin{itemize}
  \item For a geometric lamination $\Lambda^+_\phi$, the nonattracting subgroup system can be written as
  $\mathcal{A}_{na}(\Lambda^+_\phi) = \mathcal{F}\cup \{[\langle c_1 \rangle], [\langle c_2 \rangle], ..., [\langle c_k \rangle]\}$, where 
  $\mathcal{F}$ is a free factor system and $\langle c_i \rangle$ are infinite cyclic subgroups where the conjugacy classes $[c_i]$ represents the 
  boundary components of the surface which supports the lamination. This fact can be extracted from the Subgroup Decomposition work 
  \cite{HM-13a} where the geometric models are developed to study geometric strata (see Remark 1.3 in \cite{HM-13c} or proof of 
  Proposition 5.11 \cite[Case 3b, pages 49-50]{HM-09}). This allows us to do 
  induction on the rank of the components of the nonattracting subgroup system. 
  \item For a nongeometric lamination, we know that $\mathcal{A}_{na}(\Lambda^+_\phi)$  is a free factor system (item 6, Lemma \ref{NAS}).
  
 \end{itemize}

 We record this result as a corollary:
 
 \begin{corollary}
 \label{GL}
  The mapping torus of every rotationless, exponentially growing  $\phi\in\out$ is (strongly) hyperbolic relative to a finite collection of peripheral 
  subgroups of the form $F^i\rtimes_{\Phi_i} \mathbb{Z}$, where $\Phi_i$ is a lift of $\phi$ that preserves $F^i$ 
  and the outer automorphism class of $\Phi_i$ restricted to $F^i$ is polynomially growing. 
  \end{corollary}

 \textbf{Remark:} After the first version of the paper was uploaded on the arxiv, it was communicated to the author by Derrik Wigglesworth 
 that he is working on proving the Gautero-Lustig version by using the CT theory.

 \section{Polynomial growth case}
 
 If $\phi\in\out$ has at most polynomial growth, then we have the following result due to Button-Kropholler:
 
 \begin{lemma}\cite[Corollary 4.3]{BuK-16}
  
  If $\phi\in\out$ is a polynomially growing outer automorphism then the mapping torus $\F\rtimes_\Phi\mathbb{Z}$ is virtually acylindrically hyperbolic for any 
  lift $\Phi$ of $\phi$, unless $\phi$ has finite order in $\out$.
  
 \end{lemma}

 Combining the above lemma with Theorem \ref{main} we conclude that 
 
 \begin{corollary}\label{coro}
  If $\phi\in\out$ then the extension group $\F\rtimes_\Phi \mathbb{Z}$ is virtually acylindrically hyperbolic but not relatively hyperbolic 
  if and only if $\phi$ is polynomially growing and has infinite order.
 \end{corollary}
 
The ``virtually'' condition is imposed since the result of Button-Kropholler shows that one may need to pass to some power of $\phi$ to get the 
acylindrically hyperbolic conclusion and it is as open question if acylindrical hyperbolicity is quasi-isometry invariant (Mark Hagen pointed this out 
to the author).
 
 In conclusion of this section we would like to point out that this answers a question asked by Minasyan-Osin \cite[Problem 8.2]{MinOs-15}. 
 We now know that $\F\rtimes_\Phi\mathbb{Z}$ is not (virtually) acylindrically hyperbolic if and only if $\phi$ has finite order in $\out$.

 \section{Applications:}
 \subsection{Free-by-Free hyperbolic extensions}
 Consider the short exact sequence $$1\to\F\to\Gamma\to Q\to 1 $$ where $Q$ is a free subgroup of $\out$.
 In this section we shall give the construction of a free-by-free hyperbolic extension $\Gamma$ where the elements of 
 quotient group $Q$  
   are not necessarily fully irreducible. So far in the study of $\out$, the only examples of free-by-free 
 hyperbolic extensions which are known; necessarily assume that every element of $Q$ is fully irreducible. 
 So Theorem \ref{free-by-free} , in a way, gives a  new class of examples of free-by-free hyperbolic extensions.

 \textbf{Standing assumptions:}
 \label{SA}
 We let $\phi,\psi$ be outer automorphisms with dual lamination pairs $\Lambda^\pm_\phi$ and $\Lambda^\pm_\psi$ which are topmost and assume 
 that the following hold:
 \begin{enumerate}
  \item $\{\Lambda^+_\phi, \Lambda^-_\phi\} \cap \{\Lambda^+_\psi, \Lambda^-_\psi\} = \emptyset$ 
  \item $\mathcal{A}_{na}(\Lambda^\pm_\phi)$ and $\mathcal{A}_{na}(\Lambda^\pm_\psi)$ are both trivial.
  \item $\Lambda^+_\phi$ is invariant under $\phi$ and $\Lambda^-_\phi$ is invariant under $\phi^{-1}$. Similarly for 
  $\psi, \psi^{-1}$.
  
 \end{enumerate}

 \begin{lemma}\label{mutualatt}
  If $\phi,\psi$ are outer automorphisms which satisfy the standing assumptions sbove, then we have the following:
  \begin{enumerate}
   \item $\phi, \psi$ are both atoroidal. 
   \item Generic leaves of $\Lambda^+_\phi, \Lambda^-_\phi$ are attracted to $\Lambda^+_\psi$ under action of $\psi$. Similarly, 
  with roles of $\phi$, $\psi$ reversed.
  \item Generic leaves of $\Lambda^+_\phi, \Lambda^-_\phi$ are attracted to $\Lambda^-_\psi$ under action of $\psi^{-1}$. Similarly, with 
  roles of $\phi, \psi$ reversed.
  \end{enumerate}

 \end{lemma}

\begin{proof}

 \textit{Proof of (1): } Since $\mathcal{A}_{na}(\Lambda^\pm_\phi)=\emptyset$, every conjugacy class is attracted to $\Lambda^+_\phi$ 
 under the action of $\phi$. Hence $\phi$ cannot have periodic conjugacy classes. Similar argument works for $\psi$.

 \textit{Proof of (2): } Let $\gamma^+_\phi$  denote a generic leaf of $\Lambda^+_\phi$. We claim that $\gamma^+_\phi$ 
 cannot be a leaf of $\Lambda^-_\phi$. It is clear that $\gamma^+_\phi$ 
 cannot be a generic leaf of $\Lambda^-_\psi$, since otherwise the weak closure would be equal to both $\Lambda^+_\phi$ and 
 $\Lambda^-_\psi$, which would violate item (1) in the standing assumption. Also, if $H_r$ is the EG strata associated to 
 $\Lambda^-_\psi$ for some relative train track map $f:G\to G$ and if $\gamma^+_\phi$ does not have height $r$, 
 then by construction of the nonattracting subgraph we see that it is carried by $\mathcal{A}_{na}(\Lambda^\pm_\psi)$, but this is
 not possible since $\mathcal{A}_{na}(\Lambda^\pm_\psi)=\emptyset$. Hence $\gamma^+_\phi$ cannot be nongeneric leaf of $\Lambda^-_\psi$ of 
 height less than $r$. 
 If $\gamma^+_\phi$ is a nongeneric leaf of height $r$, then \cite[Lemma 3.1.15]{BFH-00} closure of $\gamma^+_\phi$ is all of $\Lambda^-_\psi$ 
 which again contradicts item (1) of the standing assumption.  
 Thus $\gamma^+_\phi$ cannot be a nongeneric leaf of $\Lambda^-_\psi$ and 
 consequently $\gamma^+_\phi$ cannot be a leaf of $\Lambda^-_\psi$.

 Choose attracting neighborhoods $V^+_\psi$ and $V^-_\psi$ of $\Lambda^+_\psi$ and $\Lambda^-_\psi$, respectively, defined by long generic leaf segments of the 
 respective laminations such that $\gamma^+_\phi\notin V^-_\psi$.
  
 By using the weak attraction theorem \ref{WAT}
 together with the standing assumption we know that either  $\gamma^+_\phi\in V^-_\psi$ or $\gamma^+_\phi$ is weakly attracted 
 to $\Lambda^+_\psi$ under iteration by $\psi$. But since we have ruled out the first possibility, $\gamma^+_\phi$  is necessarily attracted to 
 $\Lambda^+_\psi$. This implies that generic leaves of $\Lambda^+_\phi$ are attracted to $\Lambda^+_\psi$.
 
 The proof of other conclusions in item (2) and (3) follow from symmetric arguments.
\end{proof}

  \begin{definition}
  A sequence of conjugacy classes $\{\alpha_i\}$ is said to \textit{approximate} $\Lambda^+_\phi$ if 
  for any $L>0$, the ratio
  $$\frac{m(x\in S^1_i| \text{the L-nbd of x is a generic leaf segment of } \Lambda^+_\phi)}{m(S^1_i)}$$
  converges to $1$ as $i \to\infty$, where $m$ is the scaled Lebesgue measure and 
  $\tau_i:S^1_i \to G$ denotes the immersion that gives the circuit in $G$ representing 
  $\alpha_i$.
 \end{definition}

 \begin{lemma}\label{conv}
  Let $\phi, \psi$ be outer automorphisms of $\F$ which satisfy the standing assumptions above.
  Then for any sequence 
  of conjugacy classes $\{\alpha_i\}$, the sequence cannot approximate both $\Lambda^-_\phi$ and $\Lambda^-_\psi$.
 \end{lemma}

 \begin{proof}
Suppose that $\{\alpha_i\}$ approximates $\Lambda^-_\psi$. Since our standing assumptions imply (by using Lemma \ref{mutualatt}) that generic leaves of 
$\Lambda^\pm_\psi$  are weakly attracted to $\Lambda^-_\phi$ under iterations of $\phi^{-1}$, we claim that we may choose 
  an attracting neighborhood $V^-_\phi$ of $\Lambda^-_\phi$ defined by a long generic leaf segment of $\Lambda^-_\phi$ such that 
  $\Lambda^\pm_\psi\notin V^-_\phi$. If such a leaf segment does not exist then it would imply that some generic leaf of $\Lambda^-_\phi$ 
  is a leaf of either $\Lambda^+_\psi$ or $\Lambda^-_\psi$. This implies that either 
  $\Lambda^-_\phi\subset \Lambda^+_\psi$ or $\Lambda^-_\phi\subset \Lambda^-_\psi$. 
  For concreteness suppose that $\Lambda^-_\phi\subset \Lambda^+_\psi$. This implies that generic leaves of $\Lambda^-_\phi$ cannot be
  attracted to $\Lambda^-_\psi$, since otherewise the inclusion $\Lambda^-_\phi\subset \Lambda^+_\psi$ implies that 
  generic leaves of $\Lambda^+_\psi$ would get attracted to $\Lambda^-_\psi$, which is not possible.
  But then this would 
  violate that generic leaves of $\Lambda^-_\phi$ is weakly attracted to both $\Lambda^+_\psi$ and $\Lambda^-_\psi$ as in the conclusion of Lemma \ref{mutualatt}.

  Now notice that under this setup, $\Lambda^-_\psi\notin V^-_\phi$ and generic leaves of $ \Lambda^-_\psi$ are not carried by 
  $\mathcal{A}_{na}(\Lambda^\pm_\phi)$. Hence by applying the uniformity part of the weak attraction theorem, we 
  get an $M\geq 1$ such that $\phi^m_\#(\gamma^-_\psi)\in V^+_\phi$ for all $m\geq M$ for every generic leaf 
  $\gamma^-_\psi\in\Lambda^-_\psi$.  Since $V^+_\phi$ is an open set 
  we can find an $I\geq 1$ such that $\phi^m_\#(\alpha_i)\in V^+_\phi$ for all $m\geq M, i\geq I$. This 
  implies that $\{\alpha_i\}$ cannot approximate $\Lambda^-_\phi$, since otherwise generic leaves of 
  $\Lambda^-_\phi$  would get attracted to $\Lambda^+_\phi$.
 \end{proof}
 
 Now we are ready to prove a version of Mosher's 3-of-4 stretch lemma \cite{Mos-97} and give a new example of a hyperbolic-by-hyperbolic hyperbolic 
 group !
 
 \begin{proposition}[3-of-4 stretch]
 \label{34}
  Let $\phi, \psi$ be outer automorphisms which satisfy the standing assumptions above. Then we have the following:
  \begin{enumerate}
   \item There exists some $M\geq 0$ such that 
   for any conjugacy class $\alpha$,  at least three of the four numbers 
   $$ {||\phi^{n_i}_\#(\alpha)||}_{el}, {||\phi^{-n_i}_\#(\alpha)||}_{el}, 
  {||\psi^{n_i}_\#(\alpha)||}_{el}, {||\psi^{-n_i}_\#(\alpha)||}_{el}$$
  are greater an or equal to $3{||\alpha||}_{el}$, for all $n_i\geq M$.
  
  \item There exists some $N\geq 0$ such that for any word $w\in\F$, at least three of the four numbers 
   
   $$ {|\Phi^{n_i}_\#(w)|}_{el}, {|\Phi^{-n_i}_\#(w)|}_{el}, 
  {|\Psi^{n_i}_\#(w)|}_{el}, {|\Psi^{-n_i}_\#(w)|}_{el}$$
  are greater than $3{|w|}_{el}$, for all $n_i\geq N$.
   
  \end{enumerate}
 \end{proposition}
 
 \begin{proof}
  Proof of (1): 
  Suppose there does not exist any such $M_0$. We argue to a contradiction by using the weak attraction theorem. 
  By our supposition we get a sequence of conjugacy classes 
  $\alpha_i$ such that at least two of the four numbers 
  ${||\phi^{n_i}_\#(\alpha_i)||}_{el}, {||\phi^{-n_i}_\#(\alpha_i)||}_{el}$, 
  ${||\psi^{n_i}_\#(\alpha_i)||}_{el}, {||\psi^{-n_i}_\#(\alpha_i)||}_{el}$ are less than 
  $3{||\alpha_i||}_{el}$ and $n_i>i$. Proposition \ref{conjflare} tells us that at least one of 
  $\{{||\phi^{n_i}_\#(\alpha_i)||}_{el}, {||\phi^{-n_i}_\#(\alpha_i)||}_{el}\}$ is $\geq 3{||\alpha_i||}_{el}$
   and at least one of 
   
   $\{{||\psi^{n_i}_\#(\alpha_i)||}_{el}, {||\psi^{-n_i}_\#(\alpha_i)||}_{el}\}$ is $\geq 3{||\alpha_i||}_{el}$
  for all sufficiently large $i$. 
  
  For sake of concreteness suppose that 
  ${||\phi^{n_i}_\#(\alpha_i)||}_{el} \leq 3{||\alpha_i||}_{el}$ and 
  ${||\psi^{n_i}_\#(\alpha_i)||}_{el} \leq 3{||\alpha_i||}_{el}$ for all $n_i$. \hfill ($\ast$)
  
  Using Lemma \ref{conv} we know that the sequence $\{\alpha_i\}$ cannot approximate both $\Lambda^-_\phi$ and and 
  $\Lambda^-_\psi$.  For concreteness suppose that $\{\alpha_i\}$ does not approximate $\Lambda^-_\phi$.
  Then we can choose some attracting neighborhood $V^-_\phi$ of $\Lambda^-_\phi$ which is defined by some long 
  generic leaf segment and after passing to a subsequence if necessary we may assume that $\alpha_i\notin V^-_\phi$ for all 
  $i$. Also since $\mathcal{A}_{na}(\Lambda^\pm_\phi)=\emptyset$ 
  by using the uniformity part of the weak attraction theorem, there exists some $M$ 
  such that $\phi^m_\#(\alpha_i)\in V^+_\phi$ for all $m\geq M$. Choosing $i$ to be sufficiently large 
  we may assume that $n_i \geq M$ and so by using Lemma \ref{legality} we have ${LEG}_{H^\sigma_r}(\phi^{n_i}_\#(\alpha_i))\geq \epsilon$ for some $\epsilon>0$.
  By using Lemma \ref{flare} we obtain that for any $A>0$ there exists some $M_1$ such that 
  $${|\phi^m_\#(\alpha_i)|}_{\langle\mathcal{Z}, \hat{\sigma}\rangle}\geq A {|\alpha_i|}_{\langle\mathcal{Z}, \hat{\sigma}\rangle} $$ 
  for every $m>M_1$. This implies that 
  for all sufficiently large $i$, $|\phi^{n_i}_\#(\alpha_i)|_{\langle\mathcal{Z}, \hat{\sigma}\rangle} \geq A |\alpha_i|_{\langle\mathcal{Z}, \hat{\sigma}\rangle}$. Choosing a 
  sequence $A_i\to \infty$ and after passing to a subsequence of $\{n_i\}$ we may assume that 
  $|\phi^{n_i}_\#(\alpha_i)|_{\langle\mathcal{Z}, \hat{\sigma}\rangle} \geq A_i |\alpha_i|_{\langle\mathcal{Z}, \hat{\sigma}\rangle}$. 
  But this implies that the ratio 
  $|\phi^{n_i}_\#(\alpha_i)|_{\langle\mathcal{Z}, \hat{\sigma}\rangle}/|\alpha_i|_{\langle\mathcal{Z}, \hat{\sigma}\rangle} \to \infty$ as $i\to\infty$. This contradicts ($\ast$).
  
  Proof of (2) is similar to proof of Proposition \ref{strictflare}.
 \end{proof}

 \begin{theorem}\label{free-by-free}
  Let $\phi, \psi$ be outer automorphisms which satisfy the standing assumptions above. Then there exists some 
  $M>0$ such that for every $m, n \geq M$ the group $Q:=\langle \phi^m, \psi^n \rangle$ is a free group of rank $2$ and 
  the extension group $\F\rtimes \widetilde{Q}$ is hyperbolic for any lift $\widetilde{Q}$ of $Q$.
 \end{theorem}

 \begin{proof}
  The conclusion about free groups follows directly from 3-of-4 stretch result in Proposition \ref{34}. 
  The hyperbolicity of the extension group follows by using the Bestvina-Feighn Combination theorem \cite{BF-92} 
  since item (2) of Proposition \ref{34} implies that the annuli flare condition is satisfied.
 \end{proof}

 It is worth pointing out that it is very easy to construct examples of $\phi, \psi$ such that no element of $Q$ will be fully irreducible.
 \begin{corollary}\cite[Theorem 5.2]{BFH-97}
  Suppose $\phi, \psi$ are fully irreducible and atoroidal which do not have common powers.
  Then there exists some 
  $M>0$ such that for every $m, n \geq M$ the group $Q:=\langle \phi^m, \psi^n \rangle$ is a free group of rank $2$ and 
  the extension group $\F\rtimes \widetilde{Q}$ is hyperbolic for any lift $\widetilde{Q}$ of $Q$.
 \end{corollary}
\begin{proof}
 Since $\phi, \psi$ are fully irreducible, atoroidal, our standing assumptions are automatically satisfied. Now apply Theorem \ref{free-by-free}.
\end{proof}

 \textbf{Remark: } Caglar Uyanik has pointed out to the author that the above theorem can also be deduced from his work \cite{Uya-17}, since 
 our standing assumptions imply that the hypothesis of \cite[Proposition 4.2]{Uya-17} (the version of Mosher's 3-of-4 stretch lemma in his work)
 is satisfied. The attracting laminations in our hypothesis act as the attracting and repelling simplices in the space of currents 
 which is used for a ping-pong argument in that paper.
 
 \subsection{Quadratic isoperimetric inequality}
Bridson-Groves \cite{BrG-10} proved that the mapping tori of any outer automorphism of a free group satisfies the quadratic 
isoperimetric inequality. We can deduce the same theorem from our work here, and it is perhaps a simpler proof of the Bridson-Groves theorem. 
In another related work  Macura \cite{Mac-00} has some interesting results on the quadratic isoperimetric inequality problem.

\begin{theorem}\label{quad}
 The mapping torus of any $\phi\in\out$ satisfies the quadratic isoperimetric inequality.
\end{theorem}

\begin{proof}
 If $\phi$ is polynomially growing, then the result follows from Bridson-Groves theorem for the special case of polynomially growing outer automorphisms.
 Otherwise, for exponentially growing $\phi$, denote its mapping tori by $\Gamma$.
 Pass to a rotationless power, call it $\phi'$, and use Corollary \ref{GL} to conclude that the mapping torus of $\phi'$, $\Gamma'$ say, is 
 (strongly) hyperbolic relative to a collection of peripheral subgroups each of which satisfy the quadratic isoperimetric 
 inequality (by applying polynomially growth case of Bridson-Groves' theorem on the peripheral subgroups). Farb's work in \cite{Fa-98} shows that if the peripheral subgroups 
 of a relatively hyperbolic group $G$, satisfies the quadratic isoperimetric inequality, then $G$ satisfies the quadratic isoperimetric inequality. 
Hence we conclude that $\Gamma'$ satisfies the quadratic isoperimetric inequality. 

Since $\Gamma'$ is a finite index subgroup of $\Gamma$ and the property of satisfying the quadratic isoperimetric inequality is  
quasi-isometry invariant, $\Gamma$ also satisfies the quadratic isoperimetric inequality.

 \end{proof}

\noindent\rule[0.5ex]{\linewidth}{1pt}
Address: Department of Mathematical Sciences, IISER Mohali, Punjab, India.\\
Contact: \href{mailto:pritam@scarletmail.rutgers.edu}{pritam@scarletmail.rutgers.edu}

\bibliographystyle{plainnat}
\def\bibfont{\footnotesize}
\bibliography{biblo}

\begin{thebibliography}{30}
\providecommand{\natexlab}[1]{#1}
\providecommand{\url}[1]{\texttt{#1}}
\expandafter\ifx\csname urlstyle\endcsname\relax
  \providecommand{\doi}[1]{doi: #1}\else
  \providecommand{\doi}{doi: \begingroup \urlstyle{rm}\Url}\fi

\bibitem[Behrstock et~al.(2009)Behrstock, Drutu, and Mosher]{BDM-09}
J.~Behrstock, C.~Drutu, and L.~Mosher.
\newblock {Thick metric spaces, Relative hyperbolicity and Quasi-isometric
  rigidity}.
\newblock \emph{math. Annalen}, 344:\penalty0 543--595, 2009.

\bibitem[Bestvina and Feighn(1992)]{BF-92}
M.~Bestvina and M.~Feighn.
\newblock {A combination theorem for negetively curved groups}.
\newblock \emph{J. Differential Geom.}, 43\penalty0 (4):\penalty0 85--101,
  1992.

\bibitem[Bestvina and Handel(1992)]{BH-92}
M.~Bestvina and M.~Handel.
\newblock {Train tracks and {A}utomorphisms of {F}ree groups}.
\newblock \emph{Ann. of Math.}, 135:\penalty0 1--51, 1992.

\bibitem[Bestvina et~al.(1997)Bestvina, Handel, and Feighn]{BFH-97}
M.~Bestvina, M.~Handel, and M.~Feighn.
\newblock {Laminations, {T}rees, and irreducible automorphisms of free groups}.
\newblock \emph{Geom. Func. Anal.}, 7:\penalty0 215--244, 1997.

\bibitem[Bestvina et~al.(2000)Bestvina, Feighn, and Handel]{BFH-00}
M.~Bestvina, M.~Feighn, and M.~Handel.
\newblock {Tits alternative in Out($F_n$)-I: {D}ynamics of
  {E}xponentially-growing {A}utomorphisms}.
\newblock \emph{Ann. of Math.}, 151:\penalty0 517--623, 2000.

\bibitem[Bestvina et~al.(2005)Bestvina, Feighn, and Handel]{BFH-05}
M.~Bestvina, M.~Feighn, and M.~Handel.
\newblock {The Tits alternative for Out(Fn) II: A Kolchin type theorem}.
\newblock \emph{Ann. of Math.}, 161\penalty0 (2):\penalty0 1--59, 2005.

\bibitem[Bowditch(1997)]{Bow-97}
B.~Bowditch.
\newblock {Relatively hyperolic groups}.
\newblock \emph{Preprint}, Southampton, 1997.

\bibitem[Bridson and Groves(2010)]{BrG-10}
M.~Bridson and D.~Groves.
\newblock {The quadratic isoperimetric inequality for mapping tori of free
  group automorphisms}.
\newblock \emph{Mem. Amer. Math. Soc.}, 203:\penalty0 1--152, 2010.

\bibitem[Brinkmann(2000)]{Br-00}
P.~Brinkmann.
\newblock {Hyperbolic Automorphisms of Free Groups}.
\newblock \emph{Geom. Funct. Anal.}, 10\penalty0 (5):\penalty0 1071--1089,
  2000.

\bibitem[Button and Kropholler(2016)]{BuK-16}
J.~Button and R.~Kropholler.
\newblock {Nonhyperbolic free-by-cyclic and one-relator groups}.
\newblock \emph{NewYork Journal of mathematics}, 22:\penalty0 755--774, 2016.

\bibitem[Dowdall and Taylor(2018)]{DT-18}
S.~Dowdall and S.~Taylor.
\newblock {Hyperbolic extensions of free groups}.
\newblock \emph{Geom. Topol.}, 22:\penalty0 517--570, 2018.

\bibitem[Drutu(2009)]{Dr-09}
C.~Drutu.
\newblock {Relatively hyperbolic groups: geometry and quasi-isometric
  invariance}.
\newblock \emph{Comment. Math. Helv.}, 84:\penalty0 503--546, 2009.

\bibitem[Farb(1998)]{Fa-98}
B.~Farb.
\newblock {Relatively hyperolic groups}.
\newblock \emph{Geom. Funct. Anal.}, 8:\penalty0 810--840, 1998.

\bibitem[Feighn and Handel(2011)]{FH-11}
M.~Feighn and M.~Handel.
\newblock {The recognition theorem for Out($F_n)$}.
\newblock \emph{Groups Geom. Dyn}, 5:\penalty0 39--106, 2011.

\bibitem[Gautero(2016)]{Gau-16}
F.~Gautero.
\newblock {Geodesics in Trees of Hyperbolic and Relatively Hyperbolic Spaces}.
\newblock \emph{Proceedings of the Edinburgh Mathematical Society},
  59:\penalty0 701--740, 2016.

\bibitem[Gautero and Lustig(2007)]{Gau-07}
F.~Gautero and M.~Lustig.
\newblock {The mapping torus group of a free group automorphism is hyperbolic
  relative to the canonical subgroups of polynomial growth}.
\newblock \emph{Preprint}, Preprint, 2007.

\bibitem[Ghosh(2017)]{self17}
P.~Ghosh.
\newblock {Limits of conjugacy classes under iterates of hyperbolic elements of
  Out($\mathbb{F}$)}.
\newblock \emph{arXiv}, 2017.
\newblock URL \url{https://arxiv.org/abs/1709.09024}.

\bibitem[Ghosh(2018)]{self2}
P.~Ghosh.
\newblock {Relatively irreducible free subgroups of Out($\mathbb{F}$)}.
\newblock \emph{arXiv}, 2018.
\newblock URL \url{https://arxiv.org/abs/1210.8081}.

\bibitem[Hagen and Niblo(2018)]{HagNib-18}
M.~Hagen and G.~Niblo.
\newblock {A remark on thickness of free-by-cyclic groups}.
\newblock \emph{In preparation}, pages 1--5, 2018.

\bibitem[Hagen and Wise(2016)]{HagenW-16}
M.~Hagen and D.~Wise.
\newblock {Cubulating mapping tori of polynomial growth free group
  automorphisms}.
\newblock \emph{arXiv:1605.07879}, 2016.
\newblock URL \url{https://arxiv.org/abs/1605.07879}.

\bibitem[Handel and Mosher(2017{\natexlab{a}})]{HM-13a}
M.~Handel and L.~Mosher.
\newblock {Subgroup decomposition in Out($F_n$), Part I: Geometric Models}.
\newblock \emph{Mem. Amer. Math. Soc}, (accepted), 2017{\natexlab{a}}.
\newblock URL \url{http://arxiv.org/abs/1302.2378}.

\bibitem[Handel and Mosher(2017{\natexlab{b}})]{HM-13c}
M.~Handel and L.~Mosher.
\newblock {Subgroup decomposition in Out($F_n$), Part III: Weak Attraction
  theory}.
\newblock \emph{Mem. Amer. Math. Soc.}, (accepted), 2017{\natexlab{b}}.
\newblock URL \url{https://arxiv.org/abs/1306.4712}.

\bibitem[Handel and Mosher(2009)]{HM-09}
M.~Handel and Lee Mosher.
\newblock {Subgroup classification in Out($F_n$)}.
\newblock \emph{arXiv}, 2009.
\newblock URL \url{http://arxiv.org/abs/0908.1255}.

\bibitem[Macura(2000)]{Mac-00}
N.~Macura.
\newblock {Quadratic isoperimetric inequality for mapping tori of polynomially
  growing outer automorphisms of free groups}.
\newblock \emph{Geom. Funct. Anal.}, 10\penalty0 (4):\penalty0 874--901, 2000.

\bibitem[Macura(2002)]{Mak-02}
N.~Macura.
\newblock {Detour functions and quasi-isometries}.
\newblock \emph{Quat. J. of Math.}, 53:\penalty0 207--239, 2002.

\bibitem[Minasyan and Osin(2015)]{MinOs-15}
A.~Minasyan and D.~Osin.
\newblock {Acylindrical hyperbolicity of groups acting on trees}.
\newblock \emph{math. Annalen}, 362:\penalty0 1055--1105, 2015.

\bibitem[Mj and Reeves(2008)]{MjR-08}
M.~Mj and L.~Reeves.
\newblock {A combination theorem for strong relative hyperolicity}.
\newblock \emph{Geom. Topol.}, 12:\penalty0 1777--1798, 2008.

\bibitem[Mosher(1997)]{Mos-97}
L.~Mosher.
\newblock {A hyperbolic-by-hyperbolic hyperbolic group}.
\newblock \emph{Proc. Ams}, 125:\penalty0 3447--3455, 1997.

\bibitem[Sisto(2012)]{Sisto-12}
A.~Sisto.
\newblock {On metric relative hyperbolicity}.
\newblock \emph{arXiv:1210.8081}, 2012.
\newblock URL \url{https://arxiv.org/abs/1210.8081}.

\bibitem[Uyanik(2017)]{Uya-17}
C.~Uyanik.
\newblock {Hyperbolic extensions of free groups using atoroidal ping-pong}.
\newblock \emph{arXiv}, pages 1--15, 2017.
\newblock URL \url{https://arxiv.org/abs/1711.07923}.

\end{thebibliography}

\end{document}